\documentclass{amsart}

\usepackage{hyperref}
\usepackage{graphicx}
\usepackage{amsmath}
\usepackage{amsthm}
\usepackage{amssymb}
\usepackage{epstopdf}
\usepackage{color}
\usepackage{enumerate}

\newtheorem{thm}{Theorem}
\newtheorem{prop}{Proposition}

\newtheorem{lemma}{Lemma}

\newtheorem{remark}{Remark}
\newtheorem{assump}{}

\newtheorem{assumpB}{}

\newtheorem{assumpC}{}

\newtheorem{definition}{Definition}
\newtheorem{corollary}{Corollary}
\newcommand{\mathscr}{\mathcal}
\usepackage{dsfont}

\newcommand{\ito}{It\^o}

\newcommand{\frechet}{Fr\'echet }

\newcommand{\rref}[1]{(\ref{#1})}

\newcommand{\Ind}{\mathds{1}}

\newcommand{\bi}{\begin{itemize}}

\newcommand{\ei}{\end{itemize}}
\newcommand{\be}{\begin{equation}} 
\newcommand{\ee}{\end{equation}}

\newcommand{\cMFG}{\text{c-MFG}}
\newcommand{\oMFG}{\text{nc-MFG}}

\newcommand{\flows}{\text{flow maps}}
\newcommand{\Flow}{\text{Flow map}}

\renewcommand{\l}{\left}
\renewcommand{\r}{\right}
\newcommand{\la}{\l\langle}
\newcommand{\ra}{\r\rangle}

\newcommand{\R}{\mathbb{R}}

\renewcommand{\P}{\mathscr{P}}

\newcommand{\Mt}[1]{\mathscr{L}^2_{#1}(\Ptwo(\R^n))}
\newcommand{\MM}[2]{\mathscr{M}([#1,#2],\R^n)}
\renewcommand{\H}{\mathscr{H}}
\newcommand{\W}{\mathscr{W}}

\newcommand{\Ptwo}{\mathscr{P}_{2}}
\newcommand{\Wtwo}{\mathscr{W}_{2}}
\newcommand{\Htwo}[1]{\H^{2}([0,T];\mathbb{R}^{#1})}

\newcommand{\Ltwo}{\mathscr{L}^{2}}

\newcommand{\tF}{\tilde{\F}}
\newcommand{\tP}{\tilde{\mathbb{P}}}

\newcommand{\F}{\mathscr{F}}
\newcommand{\G}{\mathscr{G}}

\newcommand{\CP}[1]{\mc{M}([0,T];\Ptwo(\R^{#1}))}
\newcommand{\MP}[2]{\mc{M}([#1,#2];\Ptwo(\R^n))}

\renewcommand{\b}{\bar}

\renewcommand{\t}{\tilde}
\newcommand{\h}{\hat}
\newcommand{\tr}{\text{tr}}

\newcommand{\tsg}{\tilde{\sigma}}
\newcommand{\tSigma}{\tilde{\Sigma}}
\newcommand{\sg}{\sigma}
\newcommand{\tw}{\t{\omega}}
\newcommand{\bal}{\bar{\alpha}}
\newcommand{\ha}{\hat{\alpha}}

\newcommand{\mc}{\mathcal}
\newcommand{\mb}{\mathbb}
\newcommand{\mbb}{\mathbb}

\newcommand{\eps}{\varepsilon}

\newcommand{\tdW}{d\tilde{W}}
\newcommand{\tW}{\tilde{W}}

\newcommand{\tZ}{\tilde{Z}}
\newcommand{\tz}{\tilde{z}}

\newcommand{\ba}{\bar{\alpha}}

\newcommand{\bq}{\bar{q}}

\newcommand{\bm}{\bar{m}}

\newcommand{\bH}{\bar{H}}
\newcommand{\bb}{\bar{b}}
\newcommand{\bsg}{\bar{\sigma}}
\newcommand{\btsg}{\bar{\t{\sigma}}}
\newcommand{\D}{\Delta}

\renewcommand{\d}[1]{\nabla_{#1}}

\begin{document}
\title[FBSDE with with monotone functionals and MFGs with common noise]
{Forward-backward stochastic differential equations with monotone functionals and mean field games with common noise}

\author{Saran Ahuja}
\address{Corresponding author, Building 380, Stanford University, Stanford, CA 94305 ({\tt ssunny@stanford.edu})}

\author{Weiluo Ren}
\address{Building 380, Stanford University, Stanford, CA 94305 ({\tt weiluo@stanford.edu})}

\author{Tzu-Wei Yang}
\address{Vincent Hall, 206 Church St. SE, University of Minnesota, Minneapolis, MN 55455 
({\tt yangx953@umn.edu})}

\date{}
\keywords{forward-backward stochastic differential equations, monotone functional, mean field FBSDE with conditional law, mean field games with common noise}

\begin{abstract}
We consider a system of forward-backward stochastic differential equations (FBSDEs) with monotone functionals. We show that such a system is well-posed by the method of continuation similarly to Peng and Wu (1999) for classical FBSDEs. As applications, we prove the well-posedness result for a mean field FBSDE with conditional law and show the existence of a decoupling function. Lastly, we show that mean field games with common noise are uniquely solvable under a linear-convex setting and weak-monotone cost functions and prove that the optimal control is in a feedback form depending only on the current state and conditional law.
\end{abstract}

\maketitle


\section{Introduction}
\label{introduction}

%
  
In recent years, there has been a wide interest in the study of fully-coupled mean-field forward-backward stochastic differential equations (FBSDEs) 
of the following form
\begin{equation}
	\label{fbsde_G_intro}
	\begin{aligned}
		&dX_t = b(t,X_t,Y_t,Z_t,\mb{P}_{(X_t,Y_t,Z_t)})dt + \sigma(t,X_t,Y_t,Z_t,\mb{P}_{(X_t,Y_t,Z_t)}) dW_t \\
		&dY_t =  f(t,X_t,Y_t,Z_t,\mb{P}_{(X_t,Y_t,Z_t)})dt+Z_tdW_t \\
		&X_0 = \xi, \quad Y_T = g(X_T,\mb{P}_{X_T}),
	\end{aligned}
\end{equation}
where the coefficients $b$, $\sigma$, $f$ and $g$ depend on $\mb{P}_{(X_t,Y_t,Z_t)}$, the law of the solution $(X_t,Y_t,Z_t)$. This type of FBSDEs arises naturally 
from the mean-field type problems such as mean-field games (MFG) and mean-field type control problems (MFTC) \cite{carmona2013forward,ahuja2014}.

The well-posedness of the mean-field FBSDE (\ref{fbsde_G_intro}) is studied in \cite{carmona2013forward, carmona2013probabilistic, carmona2013fbsde, Bensoussan2015}. 
In \cite{carmona2013fbsde}, Carmona and Delarue show the existence of (\ref{fbsde_G_intro}) under a bound condition on the law argument. 
In \cite{Bensoussan2015}, the existence and uniqueness results are shown under a monotonicity condition. This monotonicity condition is motivated by the well-posedness result in 
the classical fully-coupled FBSDEs developed by \cite{peng1999,hu1990}. All these results are based on the method of continuation and the Banach fixed point theorem, 
and more importantly, they are probabilistic approaches relying on the estimates on the space of random variables.


In this paper,  we are interested in extending mean-field FBSDE \eqref{fbsde_G_intro} to a more general setting where the monotonicity property can still be applied to establish well-posed result and explore its application to the MFG model. That is, we consider the fully-coupled \textit{functional} FBSDE of the following form
\begin{equation}
	\label{eq:functional FBSDE}
	\begin{aligned}
		&dX_t = B(t, X_t, Y_t, Z_t)dt + \Sigma(t, X_t, Y_t, Z_t) dW_t \\
		&dY_t = F(t, X_t, Y_t, Z_t)dt + Z_t dW_t\\
		&X_0 = \xi, \quad Y_T = G(X_T).
	\end{aligned}
\end{equation}
Here instead of the functions of the values of $(X_t, Y_t, Z_t)$, we assume that $B$, $\Sigma$, $F$ and $G$ are \textit{functionals} of 
the square integrable random variables $X_t$, $Y_t$ and $Z_t$. This functional FBSDE includes \rref{fbsde_G_intro} as one can define a \textit{lifting} functional 
\begin{align*}
B: &[0,T]\times \Ltwo(\R^n) \times \Ltwo(\R^n) \times \Ltwo(\R^{n\times d}) &\to &\quad\Ltwo(\R^n) \\
& (t,X,Y,Z)  &\to &\quad b(t,X,Y,Z,\mb{P}_{(X,Y,Z)}) 
\end{align*}
and define similarly $\Sigma,F,G$ for $\sigma,f,g$. More importantly, as we shall discuss below, this set-up includes a mean-field FBSDE with \textit{conditional} law arising from a mean-field game with \text{common noise}, a type of model which has gained significant interest in a recent years due to its application in economic and financial modeling \cite{carmona2014commonnoise, Lacker2015, ahuja2015mean,carmona2013mean}. This lifting of a function on a law to a functional on the space of 
random variables was also discussed in \cite{cardaliaguet2010, gomes2016} where in \cite{cardaliaguet2010}, Lasry and Lion apply the lifting to define a notion of derivative in the Wasserstein space.  



This paper contributes mainly to the well-posedness theory of a general class of functional FBSDE and its applications to mean-field problems. Through a functional 
set-up, we provide several new results relating to mean-field FBSDE with conditional law and MFG with common noise. First, we show the existence and uniqueness 
result of both systems under a monotone type condition. For MFG with common noise, this result leads to what we call a \textit{weak monotonicity} condition on the cost 
functions. The weak monotonicity condition was first discussed in \cite{ahuja2014} under a simpler set-up. Here, we generalize the result further and provide a simpler proof through this functional FBSDE.  

In addition to the well-posed result, by using the conditional estimate of the solution to the functional FBSDE, we are able to prove the existence of the decoupling function of mean-field 
FBSDE with conditional law. As a corrollary, we have that the solution to MFG with common noise is in a feedback form thereby establishing its Markov 
property. The Markov property of MFG was discussed heuristically in \cite{carmona2014master} and proven in the case of no common noise in \cite{delarue2014classical}. Here, we extend the result to the case with common noise. 

Closely related to our work is a recent paper by Bensoussan, Yam, and Zhang  \cite{Bensoussan2015} where they also consider a mean-field FBSDE under monotone type conditions similar to (H2.1)-(H2.3) in Peng and Wu \cite{peng1999} and several variations. Here, our assumptions are similar to (H3.2)-(H3.3) in \cite{peng1999} as they are 
related to a stochastic control problem, or in the mean-field setting, a mean-field game. Furthermore, our results pertains mainly to its application to mean-field game model 
particularly in the case of common noise, and, thus, we consider a system with conditional law which was not discussed in \cite{Bensoussan2015}.  

The paper is organized as follows. In section \ref{sec_fbsde}, we discuss the existence and uniqueness of the solution to the functional FBSDE 
(\ref{eq:functional FBSDE}) by extending the proof in \cite{peng1999} and provides the regularity of the solution.  
In section \ref{sec_mkvfbsde}, as an application, we study a mean-field FBSDE with conditional probabilities and gives well-posedness result and the existence of a decoupling function.  The mean-field game with common noise is discussed in section \ref{sec_mfg}. Finally, the technical proofs of Theorem \ref{wellposed},
\ref{estimate_fbsde} are provided in the Appendix.

\section{FBSDE with conditional monotone functionals}
\label{sec_fbsde}

\subsection{Notations and assumptions}\label{subsec_notation}

Let $(\Omega,\F,\mb{F}=\{\F_t\}_{0 \leq t \leq T},\mb{P})$ denote a complete filtered probability space augmented by $\mb{P}$-null sets on which a $d$-dimensional 
Brownian motion $(W_t)_{0 \leq t \leq T}$ is defined. Let $\R^n$ denote the $n$-dimensional Euclidean space with the usual inner product and norm, and  
$\R^{n \times d}$ denote the Hilbert space of $(n\times d)$-matrices with inner product $\la A,B \ra = \mathbf{Tr}(A^T B)$ and the induced norm 
$|A|^2 = \mathbf{Tr}(A^TA)$.

For any sub $\sigma$-algebra $\G$ of $\F$, let $\Ltwo_{\G}(\R^k)$ denote the set of $\G$-measurable $\R^k$-valued square integrable random variables. Suppose $\mb{G} = \{ \G_t \}_{0 \leq t \leq T}$ is a sub-filtration of $\mb{F}$, then let $\H^2_\mb{G}([0,T];\R^k)$ denote the set of all $\G_t$-progressively-measurable $\R^k$-valued process $\beta = (\beta_{t})_{0 \leq t \leq T}$ such that
\begin{equation*}
	\mbb{E}\l[ \int_{0}^{T} |\beta_t|^2 dt \r] < \infty
\end{equation*}
We define similarly the space $\H_\mb{G}^2([s,t];\R^k)$ for any $0 \leq s  <  t \leq T$. We will often omit the subscript and write $\H^2([0,T];\R^k)$ for $\H^2_\mb{F}([0,T];\R^k)$. 

We consider the following FBSDE
 \begin{equation}
 \label{fbsde_G}
 \begin{aligned}
 	 &dX_t = B(t,X_t,Y_t,Z_t)dt + \Sigma(t,X_t,Y_t,Z_t) dW_t \\
	 &dY_t =  F(t,X_t,Y_t,Z_t)dt+Z_tdW_t \\
&X_0 = \xi, \quad Y_T = G(X_T)
\end{aligned}
\end{equation}
where
\begin{equation}
\begin{aligned}
&B: [0,T] \times \Ltwo_{\F}(\R^n) \times \Ltwo_{\F}(\R^n)  \times \Ltwo_{\F}(\R^{n\times d}) \to \Ltwo_{\F}(\R^n) \\
&\Sigma: [0,T] \times \Ltwo_{\F}(\R^n) \times \Ltwo_{\F}(\R^n)  \times \Ltwo_{\F}(\R^{n\times d}) \to \Ltwo_{\F}(\R^{n\times d}) \\
&F: [0,T] \times\Ltwo_{\F}(\R^n) \times \Ltwo_{\F}(\R^n)  \times \Ltwo_{\F}(\R^{n\times d}) \to  \Ltwo_{\F}(\R^n) \\
&G: \Ltwo_{\F}(\R^n) \to  \Ltwo_{\F}(\R^n)
\end{aligned}
\end{equation}
are ``functionals'' on the space of random variables and output a random variable. Our motivation for a functional set up is to solve a mean-field FBSDE similar to 
\rref{fbsde_G_intro} but with the conditional law (see \rref{mkfbsde_eq}). This type of system arises from a mean-field game with ``common noise" through the 
stochastic maximum principle. The conditional law creates certain difficulties not presented in FBSDE \rref{fbsde_G_intro}. One approach to deal with the law, 
  particularly the conditional law, is to use purely a probabilistic method. To do so, we define \textit{lifting} functionals on the space of random variables. In that 
case, we can apply the same probabilistic technique as used for a classical FBSDE, particularly those employed in \cite{peng1999}, to prove the existence, 
uniqueness, and solution estimates.

One disadvantage of using a general framework is the fact that we may lose any specific information pertaining to our system, in this case, a conditional mean-field FBSDE. To partially resolve this, we impose ``conditional'' property in the assumptions for functionals. In this way, we are able to obtain an estimate for a solution under conditional law (see Theorem \ref{estimate_fbsde}). Our main application for this result is to show existence of decoupling function of mean-field FBSDE with conditional law. This is presented in section \ref{subsec_decoupling}. 

We now state the main assumptions on the coefficients. Fix a sub-filtration $\mb{G}=\{\G_t\}_{0 \leq t \leq T}$ of $\mb{F}=\{\F_t\}_{0 \leq t \leq T}$, we assume 

\begin{assump}\label{measurable_bfg}  For $\Phi = B,F,\Sigma$, $(\Phi(t,X_t,Y_t,Z_t))_{0 \leq t \leq T}$ are $\F_t$-progressively measurable for any $(X_t,Y_t,Z_t)_{0 \leq t \leq T} \in \H^2([0,T];\R^n \times \R^n \times \R^{n \times d})$.
\end{assump}

\begin{assump}\label{bd_bfg}
\begin{equation}\label{lg_bfg}
\mb{E}\l[  \int_0^T | B(t,0,0,0)|^2 + |\Sigma(t,0,0,0)|^2+|F(t,0,0,0)|^2 dt \r]   < \infty 
\end{equation}
\end{assump}

\begin{assump}\label{lip_mon_bfg}  There exist a constant $K$ and a set of uniformly bounded linear functionals $\{c^{(1)}_t,c^{(2)}_t\}_{0 \leq t \leq T}$ where
\begin{equation*}
	c^{(1)}_t:\Ltwo_{\F}(\R^n)\to \Ltwo_{\F}(\R^k), \quad c^{(2)}_t:\Ltwo_{\F}(\R^{n \times d}) \to \Ltwo_{\F}(\R^{k})
\end{equation*}
such that for any $t \in [0,T]$, $X,X',Y,Y'\in \Ltwo_{\F}(\R^n)$, $Z,Z' \in \Ltwo_{\F}(\R^{n\times d})$, $A \in \G_t$, $\t{A} \in \G_T$, the following holds
\begin{enumerate}[(a)]
	\item $\l(c^{(1)}_t(Y_t),c^{(2)}_t(Z_t)\r)_{0 \leq t \leq T}$  are $\F_t$-progressively measurable for any $(Y_t,Z_t)_{0 \leq t \leq T} \in \H^2([0,T];\R^n \times \R^{n \times d})$.
	\item
	\label{lip_bfg}
	\begin{equation}
		\label{lip_bfg_eq}
		\begin{aligned}
			&\mb{E}\l[\Ind_A|\D B_t|^2 \r] \leq K\mb{E}\l[ \Ind_A\l(|\D X|^2 + |c^{(1)}_t(\D Y)+c^{(2)}_t(\D Z)|^2\r) \r]  \\
			&\mb{E}\l[\Ind_A|\D\Sigma_t|^2 \r] \leq K\mb{E}\l[ \Ind_A\l(|\D X|^2 + |c^{(1)}_t(\D Y)+c^{(2)}_t(\D Z)|^2\r) \r]  \\
			&\mb{E}\l[\Ind_A|\D F_t|^2 \r] \leq K\mb{E}\l[ \Ind_A\l(|\D X|^2 + |\D Y|^2+|\D Z|^2 \r) \r] \\
			&\mb{E}\l[\Ind_{\t{A}}|\D G|^2 \r] \leq K\mb{E}\l[ \Ind_{\t{A}}|\D X|^2 \r]
		\end{aligned}
	\end{equation}
	
	\item
	\label{mon_bfg}
	There exist $\beta > 0$ such that
	\begin{equation}
		\label{mon_bfg_eq}
		\begin{aligned}
			&\mb{E}\l( \Ind_A\l[ \la \D F_t, \D X \ra + \la \D B_t, \D Y \ra +  \la \D \Sigma_t, \D Z \ra \r] \r)\\
			&\qquad\leq  -\beta\mb{E}\l[ \Ind_{A}|c^{(1)}_t(\D Y)+c^{(2)}_t(\D Z)|^2 \r], \\
			&\mb{E}[ \Ind_{\t{A}}\D G\D X ] \geq 0,
		\end{aligned}
	\end{equation}
	where 
	\begin{equation*}
		\D X = X - X', \qquad \D B_t = B(t,X,Y,Z)-B(t,X',Y',Z'),
	\end{equation*}
	and $\D Y, \D Z, \D \Sigma_t, \D F_t, \D G$ are defined similarly.
\end{enumerate}
\end{assump}

The first assumption \ref{measurable_bfg} is necessary to ensure that the stochastic integral is well-defined under these functionals set up. Assumption \ref{lip_mon_bfg} is a special Lipschitz condition and monotone condition related specifically to FBSDE arising from a stochastic control problem. It is motivated by assumption (H3.2)-(H3.3) in \cite{peng1999}.

\begin{remark}\label{remark_filtration}
Note that the assumption \ref{lip_mon_bfg} depends on the filtration $\{\G_t\}_{0 \leq t \leq T}$. Thus, when it is not evident from the context, we will specify 
the filtration when referring to these assumptions. This filtration plays an important role in controlling the level of generality of our functional framework. For instance, 
if our filtration is trivial, namely $\G_t = \{ \emptyset, \Omega\}$, then the conditions are the weakest involving only on the full expectation, and so does the 
estimate of the solutions. Consequently, one cannot do much further analysis beyond the well-posedness property. On the other hands, if our filtration is too large, for instance 
$\G_t = \F$, then we can set $A=\{ X=x, Y=y, Z=z \}$ yielding a strict deterministic bound for the functionals in exchange for finer solution estimates.
\end{remark}


\subsection{Existence and uniqueness}\label{subsec_eufbsde} With the assumptions above, we have the following well-posed result.

\begin{thm}\label{wellposed} Let $\xi \in \Ltwo_{\F_s}(\R^n)$ and $B,F,G,\Sigma$ be functionals satisfying \ref{measurable_bfg}-\ref{lip_mon_bfg}, then the FBSDE \rref{fbsde_G} has the unique solution $(X_t,Y_t,Z_t)_{s \leq t \leq T}$. 
\end{thm}
 
\begin{proof} The proof for both existence and uniqueness are naturally extended from Theorem 3.1 in \cite{peng1999} for a classical FBSDE. It is based on probabilistic arguments using the method of continuation and Banach fixed point theorem on the space of square-integrable random variables. The proof is summarized in \ref{proof_wellposed}. 
\end{proof}

\subsection{Estimate}

We now give estimates of the solution to \rref{fbsde_G}. These estimates are given under conditional expectation on the same filtration specified in the assumptions. This filtration controls the level of generality of our functionals in the FBSDE (see Remark \ref{remark_filtration}). These estimates, particularly in its conditional form, will be used frequently in the subsequent sections when we discuss the existence of a decoupling function for the conditional mean-field FBSDE and the Markov property for mean-field games with common noise. 

\begin{thm}
	\label{estimate_fbsde}
	Assume that two sets of functionals $(B,\Sigma,F,G)$, $(B',\Sigma',F',G')$ satisfy \ref{measurable_bfg}-\ref{lip_mon_bfg} with the same filtration $\{\G_t\}_{s \leq t \leq T}$ and 
	$\theta_{t} = (X_{t}, Y_{t}, Z_{t})$, $\theta'_{t} = (X'_{t}, Y'_{t}, Z'_{t})$ are the solutions to the FBSDE \rref{fbsde_G} 
	with the coefficients $(B,\Sigma,F,G)$, $(B',\Sigma',F',G')$ and initial conditions $\xi, \xi' \in \Ltwo_{\F_s}(\R^n)$, respectively.  
	Then there exists a constant $C_{K, T} > 0$ depending only on $K$ and $T$ such that for any $A\in{\mathcal{G}}_{s}$,
	\begin{equation}
		\label{estimate_diff_bfg}
		\begin{aligned}
			&\mb{E} \l[ \sup_{s\leq t \leq T}\Ind_{A}|\Delta X_t|^2 + \sup_{s\leq t \leq T} \Ind_{A}|\Delta Y_t|^2 
			+ \int^{T}_{s}(\Ind_{A}|\Delta Z_t|^2)dt \r]\\ 
			&\quad \leq C_{K, T}\mb{E} \l[ \Ind_{A} \l( |\Delta \xi|^{2} +|\bar{G}|^{2} + \int^{T}_{s}(|\bar{F}_{t}|^{2} + |\bar{B}_{t}|^{2} 
			+ |\bar{\Sigma}_{t}|^{2} \r) \r] 
		\end{aligned}
	\end{equation}
	where $\Delta X_t = X_t-X'_t$ and $\Delta Y_t, \Delta Z_t, \Delta \xi$ are defined similarly,  $\bar{\Phi}_{t}= \Phi(t, \theta'_{t}) -  \Phi'(t, \theta'_{t})$ for $\Phi = B, \Sigma, F$ and $\bar{G} = G(X'_T) - G'(X'_T)$.
\end{thm}

\begin{proof} See \ref{proof_estimate_diff}.
\end{proof}

\begin{corollary}
	\label{est_functional}
	Let $(X_t,Y_t,Z_t)_{s \leq t \leq T}$ be the solution to FBSDE \rref{fbsde_G} with the initial condition $\xi \in \Ltwo_{\F_s}(\R^n)$ and coefficients $(B,\Sigma,F,G)$ satisfying \ref{measurable_bfg}-\ref{lip_mon_bfg}.  
	Then there exists a constant $C_{K, T} > 0$ depending only on $K$ and $T$ such that for any $A\in{\mathcal{G}}_{s}$,
	\begin{equation}
		\label{est_functional_eq}
		\begin{aligned}
			&\mb{E}\l[ \Ind_A \sup_{s \leq t \leq T} |X_t|^2 +  \Ind_A\sup_{s \leq t \leq T} |Y_t|^2  +  \Ind_A\int_s^T |Z_t|^2 dt \r] \\
			&\quad \leq  C_{K,T}\l( \mb{E}\l[\Ind_A |\xi|^2+ \Ind_A|G(0)|^2\r] \r)\\
			&\quad\quad + C_{K,T}\mb{E}\l[\Ind_A \int_s^T \l(|B(t,0,0,0)|^2 + |F(t,0,0,0)|^2+|\Sigma(t,0,0,0)|^2  \r) dt \r]
		\end{aligned}
	\end{equation}
\end{corollary}
\begin{proof}
	Apply Theorem \ref{estimate_fbsde} with $\xi'\equiv0$ and $(B',\Sigma',F',G')\equiv(0,0,0,0)$.
\end{proof}

\section{Mean-field FBSDE with conditional law}
\label{sec_mkvfbsde}


In this section, we discuss an application of the results on the functional FBSDE to a class of mean-field FBSDE with conditional law. In addition, we show the existence of a deterministic decoupling function for this type of FBSDE.

\subsection{Problem formulation} Following similar notations as defined in section \ref{subsec_notation}, we consider the following system

\begin{equation}
	\label{mkfbsde_eq}
	\begin{aligned}
		dX_{t} &= b(t, X_{t}, Y_{t}, Z_{t},\tilde{Z}_{t},\mb{P}_{(X_t, Y_t, Z_t , \tZ_t)|\tF_t})dt\\
		&\quad + \sigma(t, X_{t}, Y_{t}, Z_{t}, \tilde{Z}_{t}, \mb{P}_{(X_t, Y_t, Z_t , \tZ_t)|\tF_t})dW_{t} \\
		&\quad + \tsg(t, X_{t}, Y_{t}, Z_{t}, \tilde{Z}_{t}, \mb{P}_{(X_t, Y_t, Z_t , \tZ_t)|\tF_t})d\tilde{W}_{t} \\
		dY_{t} &= f(t, X_{t},  Y_{t}, Z_{t}, \tilde{Z}_{t}, \mb{P}_{(X_t, Y_t, Z_t , \tZ_t)|\tF_t})dt + Z_{t}dW_{t} + \tilde{Z}_{t}d\tilde{W}_{t}\\
		X_{0} &= \xi, \quad Y_{T} = g(X_{T}, \mb{P}_{X_T|\tF_T})
	\end{aligned}
\end{equation}
where $(W_t)_{0 \leq t \leq T}, (\tW_t)_{0 \leq t \leq T}$ are independent Brownian 
motions in $\R^{d_1},\R^{d_2}$, $\mb{P}_{(X_t, Y_t, Z_t , \tZ_t)|\tF_t}$ denotes the 
law of $(X_{t}, Y_{t}, Z_{t} , \tZ_{t})$ conditional on $\tF_t$, where $\tF_t$ denotes 
the $\sigma$-field generated by $\{\tW_s; 0 \leq s \leq t \}$. As we shall see in the 
next section, this FBSDE is related to mean-field games model with common noise. For 
this reason, we will refer to $(W_t)_{0 \leq t \leq T}$ as \textit{individual} noise, 
$(\tW_t)_{0 \leq t \leq T}$ as \textit{common} noise and $(\tF_t)_{0 \leq t \leq T}$ 
as the \textit{common noise filtration}. The tuple $(X_{t}, Y_{t}, Z_{t} , \tZ_{t})_{0 \leq t \leq T} $ is called a solution to \rref{mkfbsde_eq} if it is in $\Htwo{p}$, 
where $\R^p = \R^n \times \R^n \times \R^{n \times d_1} \times \R^{n \times d_2}$, and 
satisfies \rref{mkfbsde_eq}. Note that when $(X_t,Y_t,Z_t,\tZ_t)_{0 \leq t \leq T} \in \Htwo{p}$, the existence of $\tF_t$-progressively measure version of $(\mb{P}_{(X_t, Y_t, Z_t , \tZ_t)|\tF_t})_{0 \leq t \leq T}$ is guaranteed by Lemma 1.1 in \cite{kurtz1988} for instance.

The functions in \rref{mkfbsde_eq} are given and defined on the following spaces;
\begin{equation}
\begin{aligned}
& b,f: [0,T]\times\R^p \times \Ptwo(\R^p) \to \R^n, & \sigma: [0,T]\times\R^p \times \Ptwo(\R^p) \to \R^{n \times d_1} \\
& \tsg:[0,T]\times \R^p \times \Ptwo(\R^p) \to \R^{n \times d_2}, & g: \R^n \times \Ptwo(\R^n) \to \R^{n}
\end{aligned}
\end{equation}
where $\Ptwo(\R^d)$ denotes the space of Borel probability measures on $\R^d$ with finite second moment, i.e. a probability measure $\mu$ such that  $\int_{\R^d} x^2 d\mu(x) < \infty $. It is a complete separable metric space equipped with a second-order Wasserstein metric $\Wtwo(\cdot,\cdot)$ defined as
\be\label{wass}
 \Wtwo(\mu,\nu) = \inf_{\gamma \in \Gamma_{\mu,\nu}} \l( \int_{\R^d \times \R^d} |x-y|^2 \gamma(dx,dy) \r)^{\frac{1}{2}}
 \ee
where $ \Gamma_{\mu,\nu}$ denotes the space of probability measure on $\R^d \times \R^d$ with marginal $\mu,\nu$ respectively.

For the conditional probability flows, we introduce the space  $\H^2_{\mb{G}}([0,T];\Ptwo(\R^d))$ for all the $\G_t$-progressively-measurable $\Ptwo(\R^d)$-valued processes $(m_{t})_{0 \leq t \leq T}$ such that 
\begin{equation*}
	\mb{E}\l[ \int_0^T \int_{\R^d} |x|^{2} m_{t}(dx) dt \r] < \infty
\end{equation*}
We will mainly be interested in $\Ltwo_{\t{\mb{F}}}([0,T];\Ptwo(\R^d))$ where $\t{\mb{F}}=\{\tF_t\}_{0 \leq t \leq T}$ with $\tF_t=\sigma(\tW_s)_{0 \leq s \leq t}$ being the common noise filtration.
Note that when $(X_t,Y_t,Z_t,\tZ_t)_{0 \leq t \leq T}$ $\in \Htwo{p}$, we have $(\mb{P}_{(X_t,Y_t,Z_t,\tZ_t | \tF_t)})_{0 \leq t \leq T}\in \H^2_{\t{\mb{F}}}([0,T];\Ptwo(\R^p)) $. 

In order to construct conditional expectation given common noise explicitly, we separate the path space for individual noise and common noise. From now and throughout this section, we assume that $(\Omega,\F,\mb{P})$ is in the form $ (\Omega^0 \times \t{\Omega},\F^0 \otimes \tF,\mb{P}^0 \otimes \tP)$ where the individual noise $W_t$ and common noise $\tW_t$ are supported in the space $ (\Omega^0,\F^0 ,\mb{P}^0 )$ and $(\t{\Omega},\tF,\tP)$ respectively. We will also assume that $(\t{\Omega},\tF,\tP)$ is the canonical sample space of the Brownian motion $(\tW_t)_{0 \leq t \leq T}$ with $\tF$ being its natural filtration completed with $\mb{P}$-null sets. We also assume that $\Omega^0$ is sufficiently rich (Polish and atom-less) that for any $m \in \Ptwo(\R^n)$, we can find $\xi \in \Ltwo(\Omega^0;\R^n)$ independent of all Brownian motions with law $m$. We first provide the following lemma which will be proved useful in a subsequent section.
\begin{lemma}\label{representation_m_xi} Consider a Polish and atom-less probability space $\Omega$. For two measures $m_{1}, m_{2}\in\Ptwo(\R^d)$ satisfying $\W_2(m_{1}, m_{2})< \eps$ and any random variable $\xi \in \Ltwo(\Omega;\R^d)$ with law $m_{1}$, there exist $\eta \in \Ltwo(\Omega;\R^d)$ such that $\eta$ has law $m_{2}$ and  $(\mb{E} |\xi - \eta|^{2}|)^{\frac{1}{2}} < \eps$.
\end{lemma}
\begin{proof}
By the definition of Wasserstein metric, there exists a small enough $\eps' > 0$ and two random variables $X$ and $Y$ with law $m_{1}$ and $m_{2}$ respectively such that
\begin{equation*}
	(\mb{E} |X - Y|^{2})^{\frac{1}{2}} < \eps - \eps'
\end{equation*}

Now by Lemma 6.4 in \cite{cardaliaguet2010}, there exists a bijective mapping $\tau:\Omega \to \Omega$ that is measurable, measure-preserving, and satisfies 
\begin{equation*}
	(\mb{E} |X \circ \tau - \xi|^{2})^{\frac{1}{2}} \leq |X \circ \tau- \xi|_{\infty} < \eps'
\end{equation*}
Since $\tau$ is measure-preserving, $Y \circ \tau $ also has distribution $m_{2}$ and
\begin{equation*}
	(\mb{E} |X \circ \tau - Y \circ \tau |^{2} )^{\frac{1}{2}} = (\mb{E} |X - Y|^{2} )^{\frac{1}{2}}  < \eps - \eps' \quad \Rightarrow \quad (\mb{E} |\xi- Y \circ \tau |^{2})^{\frac{1}{2}} < \eps.
\end{equation*}
Thus, we can use $Y \circ \tau $ as our $\eta$.
\end{proof}

\subsection{Existence and uniqueness of a solution and its estimate.} Note that the coefficients $b,\sigma,\tsg,f,g$, are functionals of the law in their last arguments. To prove existence and uniqueness result for system \rref{mkfbsde_eq}, one approach is to employ the Schauder fixed point theorem. That is, we fix a flow of probability measures $(m_t)_{0 \leq t \leq T}$, replace the probability measure arguments in \rref{mkfbsde_eq} by $m_t$, solve a classical system of FBSDE, and consider the law of the solution, namely $\mb{P}_{(X_{t}, Y_{t}, Z_{t} , \tZ_{t})|\tF_t}$. This map can be described as follows;
\begin{equation}\label{diagram1}
 (m_t)_{0 \leq t \leq T} \quad\Rightarrow\quad  (X_{t}, Y_{t}, Z_{t} , \tZ_{t})_{0 \leq t \leq T} \quad\Rightarrow\quad  (\mb{P}_{(X_{t}, Y_{t}, Z_{t} , \tZ_{t})|\tF_t})_{0 \leq t \leq T} 
 \end{equation}
The fixed point of these operations then yields the solution to \rref{mkfbsde_eq}. This set up was used in \cite{carmona2013probabilistic,carmona2013mean} to prove well-posedness of \rref{mkfbsde_eq} without conditional law. However, it is considerably more difficult to extend the result using this argument to the case of conditional law since it involves the space of \textit{stochastic} flow of probability measures as opposed to the \textit{deterministic} one. In this larger space, it is non-trivial how one can find an invariant compact subset so that the Schauder fixed point theorem can be applied. 

Instead, we adopt a different approach to deal with the probability measure terms; we consider a \textit{lifting} from the space of probability measure to the space of random variables. That is, for a function $\phi: \R^p \times \Ptwo(\R^p) \to \R^d$, we define a functional $\Phi: \R^p \times \Ltwo_{\F}(\R^p) \to \R^d$ by 
\begin{equation*}
	\Phi(x,\xi) = \phi(x,\mb{P}_\xi)
\end{equation*}
This lifting allows us to work on the Hilbert space $\Ltwo_{\F}(\R^p)$ instead of the metric space $\Ptwo(\R^p)$. This approach was used in \cite{cardaliaguet2010} (see Ch.6) to define a derivative in the Wasserstein space $\Ptwo(\R^p)$ through the \frechet derivative in the Hilbert space $\Ltwo_{\F}(\R^p)$.
For the system \rref{mkfbsde_eq}, we can extend the lifting further and combine all the state variables by defining $B: [0,T] \times \Ltwo_{\F}(\R^p) \to \Ltwo_{\F}(\R^n)$ as 
\begin{equation}\label{lifting}
  B(t,X,Y,Z,\tZ) = b(t,X,Y,Z,\tZ,\mb{P}_{(X, Y, Z , \tZ)|\tilde{\mathcal{F}}_{t}})
\end{equation}
for any random variables $(X,Y,Z,\tZ) \in \Ltwo_{\F}(\R^p)$. We define similarly, $\Sigma,\tSigma,F,G$ the lifting functionals of $\sigma,\tsg,f,g$. 

Using these functionals, FBSDE \rref{mkfbsde_eq} is translated to a functional FBSDE \rref{fbsde_G}. Thus, if $B,(\Sigma,\tSigma),F,G$ defined above satisfy \ref{measurable_bfg}-\ref{lip_mon_bfg}, then we can apply our results from section \ref{sec_fbsde}, namely the Theorems \ref{wellposed} and \ref{estimate_fbsde}, to obtain the well-posedness of \rref{mkfbsde_eq} and its solution estimate. To do so, we assume the following on $b,\sigma,\tsg,f,g$. 

\begin{assumpB}\label{measurable_mkv} $b,\sigma,\tsg,f,g$ are measurable and satisfies 
\begin{equation} 
\int^{T}_{0}  |\phi(t, 0, 0, 0, 0, \delta_{0})|^{2}dt <\infty \quad\text{for }\phi=b,\sigma,\tsg,f,g
\end{equation}
where $\delta_{a}$ denotes the Dirac measure at $a \in \R^p$.
\end{assumpB}

\begin{assumpB}\label{lip_mon_mkv} There exist a constant $K$ and uniformly bounded linear maps
\begin{equation*}
	c^{(1)}_t: \R^n\to \R^k, \quad c^{(2)}_t: \R^{n \times d_1} \to \R^k, \quad c^{(3)}_t: \R^{n \times d_2} \to \R^k
\end{equation*}
such that for any $t \in [0,T]$, $\theta=(X,Y,Z,\tZ)$, $\theta'=(X',Y',Z',\tZ') \in \Ltwo((\h{\Omega}$,  $\h{\F}, \h{\mb{P}});\R^p)$, where $(\h{\Omega},  \h{\F}, \h{\mb{P}})$ is an arbitrary probability space, the following holds:
\begin{enumerate}[(a)]
	\item
	\label{lip_mkv_a}
	For $\phi= b, \sigma,\tsg,f$
	\begin{align*}
		&\h{\mb{E}}| \Delta \phi_t |^{2} \leq K \h{\mb{E}} (|\Delta X|^{2}+ |c^{(1)}_t(\Delta Y)+c^{(2)}_t(\Delta Z) + c^{(3)}_t(\Delta \tZ')|^{2} ), \\
		&\h{\mb{E}}| \Delta g |^2 \leq K  \h{\mb{E}}|\Delta X|^{2}
	\end{align*}

	\item
	\label{mon_mkv_b}
	There exists $\beta > 0$ such that
	\begin{align*}
		&\h{\mb{E}}\l[ \la \D f_t, \D X \ra + \la \D b_t, \D Y \ra +\la \D \sigma_t, \D Z \ra +   \la \D\tsg_t, \D \tilde{Z} \ra \r]\\
		&\quad \leq -\beta\h{\mb{E}}|c^{(1)}_t(\Delta Y)+c^{(2)}_t(\Delta Z) + c^{(3)}_t(\Delta \tZ')|^{2}, \\
		&\h{\mb{E}}[\D g\D X ] \geq 0
	\end{align*}
	where $\D X = X-X'$, $\Delta f_t = f(t,\theta, \h{\mb{P}}_{\theta}) - f(t,\theta', \h{\mb{P}}_{\theta'})$  and similarly for other terms.
\end{enumerate}
\end{assumpB}

It is worth noting the difference between the functionals $B,\Sigma, \tSigma, F, G$ discussed in section \ref{subsec_notation} and those on the functions $b,\sigma,\tsg, f,g$. Here, the functions $b,\sigma,\tsg, f,g$ are all deterministic and their definitions and assumptions \ref{measurable_mkv}-\ref{lip_mon_mkv} do not depend on the probabilistic setup of our model. That is, it is independent of the fixed probability space $(\Omega, \F, \mb{P})$ and/or any filtration, and particularly does not involve the conditional law.

Assumptions \ref{measurable_mkv}-\ref{lip_mon_mkv} are comparable to those in Bensoussan et al. \cite{bensoussan2014mean} where they consider an extension of assumptions (H2.1)-(H2.3) in Peng and Wu \cite{peng1999} to include mean-field terms. In our case, we give a special type of Lipschitz condition involving both $Y,Z$ simultaneously (see \ref{lip_mon_mkv}). This is comparable to assumption (H3.2)-(H3.3) in Peng and Wu \cite{peng1999} where they consider this particular case for its application to stochastic control problems. Similarly, we are interested mainly in its application to a mean-field game model which is a control problem with mean-field interaction. 

For our applications in the subsequent sections, we will state the well-posedness result for FBSDE over the time interval $[s,T]$ and slightly more general filtrations. For that, we define the following 

\begin{definition}
	\label{def_fbsde} 
	Let $s \in [0,T]$, $\xi \in \Ltwo_{\F_s}(\R^n)$ and $\{\G_t\}_{s \leq t \leq T}$ be a sub-filtration of $\{\F_t\}_{s \leq t \leq T}$. 
	We define \textit{FBSDE with data $(s,\xi,\{\G_t\}_{s \leq t \leq T})$} or simply \textit{FBSDE $(s,\xi,\{\G_t\}_{s \leq t \leq T})$} to be the following FBSDE
	\begin{equation}
		\label{fbsde_gen}
		\begin{aligned}
			dX_{t} &= b(t, X_{t}, Y_{t}, Z_{t},\tilde{Z}_{t},\mb{P}_{(X_{t}, Y_{t}, Z_{t} , \tZ_{t})|\G_t})dt\\
			&\quad + \sigma(t, X_{t}, Y_{t}, Z_{t}, \tilde{Z}_{t}, \mb{P}_{(X_{t}, Y_{t}, Z_{t} , \tZ_{t})|\G_t})dW_{t} \\
			&\quad+ \tsg(t, X_{t}, Y_{t}, Z_{t}, \tilde{Z}_{t}, \mb{P}_{(X_{t}, Y_{t}, Z_{t} , \tZ_{t})|\G_t})d\tilde{W}_{t} \\
			dY_{t} &= f(t, X_{t},  Y_{t}, Z_{t}, \tilde{Z}_{t}, \mb{P}_{(X_{t}, Y_{t}, Z_{t} , \tZ_{t})|\G_t})dt + Z_{t}dW_{t} + \tilde{Z}_{t}d\tilde{W}_{t}\\
			X_{s} &= \xi, \quad Y_{T} = g(X_{T}, \mb{P}_{X_{T}|\G_T})
		\end{aligned}
	\end{equation}
$\theta_{t} = (X_t,Y_t,Z_t,\tZ_t)_{s \leq t \leq T}$ is called a solution to FBSDE $(s,\xi,\{\G_t\}_{s \leq t \leq T})$ if they are $\F_t$-adapted and satisfy FBSDE \rref{fbsde_gen}.
\end{definition}

We are interested particularly in the FBSDE $(s,\xi,\{\tF^s_t\}_{s \leq t \leq T})$, where $\xi \in \Ltwo_{\F_s}(\R^n)$ and $\{\tF^s_t\}_{s \leq t \leq T}$ denotes 
the $\sigma$-algebra generated by the common noise starting at time $s$, i.e. $\tF^s_t = \sigma(\tW_r-\tW_s; s \leq r \leq t )$. In section \ref{subsec_decoupling}, 
we will discuss the use of this sub-FBSDE to define the so-called ``decoupling" function for mean-field FBSDE. First, we state our main result for this subsection 
which establishes their existence and uniqueness results.



\begin{thm}\label{wellposed_mckean} Assume that \ref{measurable_mkv}-\ref{lip_mon_mkv} hold, then FBSDE $(s,\xi,\{ \tF^s_t \}_{s \leq t \leq T})$ admits a unique solution $(X_{t},Y_{t},Z_{t},\tZ_{t})_{s \leq t \leq T}$ satisfying 
\begin{equation}\label{estimate_mkv}
\begin{aligned}
&\mb{E}\l[ \Ind_A \sup_{s \leq t \leq T} |X_t|^2 +  \Ind_A\sup_{s \leq t \leq T} |Y_t|^2  +  \Ind_A\int_s^T (|Z_t|^2 + |\tZ_t|^2) dt \r] \\
	&\quad \leq  C_{K,T} \mb{E}\Big[ \Ind_A |\xi|^2+ \Ind_A|g(0,\delta_0)|^2 + \Ind_A \int_s^T \Big(|b(t,0,0,0,0,\delta_0)|^2 \\
	&\quad\qquad  + |f(t,0,0,0,0,\delta_0)|^2  + \sigma(t,0,0,0,0,\delta_0)^2+\tsg(t,0,0,0,0,\delta_0)^2  \Big) dt \Big]
\end{aligned}
\end{equation}
for $A \in \F_s$, where $\delta_a$ denotes the Dirac measure at $a \in \R^p$. Moreover, for $i=1,2$, let $(X^i_t,Y^i_t,Z^i_t,\tZ^i_t)_{s \leq t \leq T}$ denote the 
solution to FBSDE  $(s,\xi^i,\{ \tF^s_t \}_{s \leq t \leq T})$, where $\xi^i \in \Ltwo_{\F_s}(\R^n)$, then the following estimate holds
\begin{equation}
	\label{estimate_diff_mkv}
	\begin{aligned}
		&\mb{E}\l[ \Ind_A \sup_{s \leq t \leq T} |\D X_t|^2 +  \Ind_A\sup_{s \leq t \leq T} |\D Y_t|^2  +  \Ind_A\int_s^T |\D Z_t|^2 + |\D \tZ_t|^2 dt \r]\\
		&\quad \leq  C_{K,T} \mb{E}\l[ \Ind_A |\D \xi|^2 \r] 
	\end{aligned}
\end{equation}
where $\D X_t = X^1_t-X^2_t$ and $\D Y_t, \D Z_t, \D \tZ_t, \D \xi$ are defined similarly. 
\end{thm}

\begin{proof} Let $B,\Sigma,\tSigma,F,G$ be lifting functionals of $b,\sigma,\tsg,f,g$ as defined in \rref{lifting} but with respect to $\{\tF^s_t\}_{s \leq t \leq T}$; that is,
\begin{equation*}
	B(t,X,Y,Z,\tZ) = b(t,X,Y,Z,\tZ,\mb{P}_{(X, Y, Z , \tZ)|\tF^s_t})
\end{equation*}
for any random variables $(X,Y,Z,\tZ) \in \Ltwo_{\F}(\R^p)$. The functionals $\Sigma,\tSigma,F,G$ are defined similarly. We need to verify that $B,(\Sigma,\tSigma),F,G$ satisfy \ref{measurable_bfg}-\ref{lip_mon_bfg} with respect to $\{\tF^s_t\}_{0 \leq t \leq T}$, then the result follows directly from Theorem \ref{wellposed} and \ref{estimate_fbsde}.  

Since the map $\H^2([s,T];\R^p) \ni  (\theta_t)_{s \leq t \leq T} \to \mb{P}_{\theta_t | \tF^s_t} \in \Ltwo_{\t{\mb{F}}}([s,T];\Ptwo(\R^p))$ is continuous, assumption \ref{measurable_bfg} follows from \ref{measurable_mkv} and so does \ref{bd_bfg}. Assumption \ref{lip_mon_bfg} follows from \ref{lip_mon_mkv}, that is,
\be
\begin{aligned}
	&\mb{E}\l[ \Ind_A |\Delta B_t| \r] = \mb{E}\l[  \mb{E} \l[ \Ind_A  |\Delta b_t|  | \tF^s_t \r] \r] = \mb{E}\l[  \Ind_A  \mb{E} \l[ |\Delta b_t|  | \tF^s_t \r] \r] \\
	&\quad \leq \mb{E}\l[ \Ind_A K \mb{E}\l[ |\Delta X|^{2}+ |c^{(1)}_t(\Delta Y)+c^{(2)}_t(\Delta Z) + c^{(3)}_t(\Delta \tZ')|^{2} | \tF^s_t \r] \r] \\
	&\quad =  K \mb{E}\l[ \Ind_A\l(  |\Delta X|^{2}+ |c^{(1)}_t(\Delta Y)+c^{(2)}_t(\Delta Z) + c^{(3)}_t(\Delta \tZ')|^{2} \r) \r]
\end{aligned}
\ee
Other conditions in \ref{lip_mon_bfg} follow similarly. 



\end{proof}

\subsection{Decoupling function of a mean-field FBSDE with conditional law}\label{subsec_decoupling}


In this section, we discuss the existence of a \textit{decoupling} function for mean-field FBSDE with conditional law. A decoupling function is a function which helps to ``decouple'' the FBSDE by describing the relation of the backward process $Y_t$ as a function of the forward process $X_t$. As a result, it reduce the FBSDE to merely solving a standard forward SDE. This method of solving FBSDE is called \textit{Four-steps scheme} and was first proposed by Ma, Protter, and Yong in \cite{ma1994solving} for a classical FBSDE with non-random coefficients. In that case, under regularity assumptions on the coefficients, one can find a decoupling function by solving a quasilinear PDE. When the coefficients are random, the decoupling function is also random and is referred to as a \textit{decoupling field} and is related to backward stochastic differential equation (BSDE). We refer to \cite{delarue2002existence} for more detail on a decoupling function of a classical FBSDE in a deterministic case and \cite{ma2011} for a decoupling field in a general case. 

Going back to our setting, we consider first the mean-field FBSDE \rref{mkfbsde_eq} with unconditional law; suppose that we fix a \textit{deterministic} flow of probability measure 
$m=(m_t)_{0 \leq t \leq T} \in  \CP{p}$ and consider the system
\begin{equation}
	\label{mkfbsde_eq_fix_m}
	\begin{aligned}
		dX_{t} &= b(t, X_{t}, Y_{t}, Z_{t},\tilde{Z}_{t},m_t)dt + \sigma(t, X_{t}, Y_{t}, Z_{t}, \tilde{Z}_{t}, m_t)dW_{t}\\
		&\quad + \tsg(t, X_{t}, Y_{t}, Z_{t}, \tilde{Z}_{t}, m_t)d\tilde{W}_{t} \\
		dY_{t} &= f(t, X_{t},  Y_{t}, Z_{t}, \tilde{Z}_{t}, m_t)dt + Z_{t}dW_{t} + \tilde{Z}_{t}d\tilde{W}_{t}\\
		X_{0} &= \xi, \quad Y_{T} = g(X_{T}, m^{(n)}_T)
	\end{aligned}
\end{equation}
where $m^{(n)}_T$ denotes the marginal distribution of the first $n$-dimension of $\R^p$ of $m_T$. Since $(b,\sigma,\tsg,f)(t,x,y,z,m_t)$, $g(x,m^{(n)}_T)$ are deterministic functions, we have, under certain standard assumptions, an existence of a decoupling function for classical FBSDEs; that is, there exist a function $U^m:[0,T] \times \R^n$ such that
\begin{equation*}
	Y_t = U^m(t,X_t)
\end{equation*}
See \cite{delarue2002existence} for instance. A Markov property for \rref{mkfbsde_eq_fix_m} would mean $U^m(t,x)$ can be written as $\b{U}(t,x,m_t)$; consequently, going back to the mean-field FBSDE \rref{mkfbsde_eq} (with unconditional law), we have
\begin{equation*}
	Y_t = U^m(t,X_t)  = \b{U}(t,X_t,m_t) = \b{U}(t,X_t,\mb{P}_{(X_t,Y_t,Z_t,\tZ_t)})
\end{equation*} 
In addition, its Markov property also means that the law of the backward processes also depends only on the law of the forward process.  As a result, a decoupling function for mean-field FBSDE \rref{mkfbsde_eq} is expected to be a deterministic function $U:[0,T] \times \R^n \times \Ptwo(\R^n)$ such that  
\begin{equation*}
	Y_t = U(t,X_t,\mb{P}_{X_t})
\end{equation*}
We note that the relation above does not follow directly from results for a classical FBSDE and the derivation above is merely heuristic. The decoupling function for unconditional mean-field FBSDE was discussed in \cite{carmona2013forward,carmona2014master} and shown 
rigorously in \cite{delarue2014classical}. In our case where the law is conditional, the flow $m$ is in fact 
\textit{stochastic} which introduces more difficulties. First, the coefficients (given $m$) $(b,\sigma,\tsg,f)(t,x,y,z,m_t), g(x,m^{(n)}_T)$ are now random, so the 
result from classical FBSDEs does not even apply directly in the first place. Secondly, the conditional law makes it difficult to deal with a time-varying FBSDE used to define the decoupling function.


However, as the coefficients in \rref{mkfbsde_eq} are still deterministic functions of $m$, if the system is well-posed, it is reasonable to expect a Markov property with respect to the conditional law of $X_t$; that is, there exist a deterministic function $U:[0,T] \times \R^n \times \Ptwo(\R^n)$ such that
\begin{equation}\label{decoupling_u}
	Y_t = U(t,X_t,\mb{P}_{X_t | \tF_t})
\end{equation}

\begin{remark} 
	In the classical FBSDE, the decoupling function of the FBSDE corresponding to a control problem is the gradient (in state variable $x$) of the value 
	function which is a solution to an HJB equation. Similarly, the decoupling function $U(t,x,m)$ here is the gradient (in $x$) of the \textit{generalized} value 
	function which satisfies the so-called \textit{master equation}. We refer to \cite{ahuja2015mean,bensoussan2014master,carmona2014master} for more detailed 
	discussion on the master equation.
\end{remark}

To state our main result showing the existence of $U$ satisfying  \eqref{decoupling_u}, we first list additional assumptions
\begin{assumpB}\label{onlyxy_mkv}
The functions $b,\sigma,\tsg,f$ in the FBSDE \rref{mkfbsde_eq} depend only on the conditional law of $(X_t,Y_t)$; that is, the FBSDE $(s,\xi,\{\G_t\}_{s \leq t \leq T})$ is now given by 
\begin{equation}
	\label{fbsde_gen_2}
	\begin{aligned}
	dX_{t} &= b(t, X_{t}, Y_{t}, Z_{t},\tilde{Z}_{t},\mb{P}_{(X_{t}, Y_{t})|\G_t})dt 
	+ \sigma(t, X_{t}, Y_{t}, Z_{t}, \tilde{Z}_{t}, \mb{P}_{(X_{t}, Y_{t})|\G_t})dW_{t} \\
	&\quad+ \tsg(t, X_{t}, Y_{t}, Z_{t}, \tilde{Z}_{t}, \mb{P}_{(X_{t}, Y_{t})|\G_t})d\tilde{W}_{t} \\
	dY_{t} &= f(t, X_{t},  Y_{t}, Z_{t}, \tilde{Z}_{t}, \mb{P}_{(X_{t}, Y_{t})|\G_t})dt + Z_{t}dW_{t} + \tilde{Z}_{t}d\tilde{W}_{t}\\
	X_{s} &= \xi, \quad Y_{T} = g(X_{T}, \mb{P}_{X_{T}|\G_T})
\end{aligned}
\end{equation}
\end{assumpB}
\begin{assumpB}\label{linearGrowthOnM_m_given}
For any $m\in\mathscr{P}_{2}(\R^{2n})$, $\t{m} \in \P_2(\R^n) $, $t\in[0,T]$, $\phi = b, \sigma, \t\sigma, f$,
\begin{equation*}
\begin{aligned}
	&\phi(t, 0, 0, 0, 0, m)^{2} \leq  K \l(1 + \int_{\R^{2n}} y^2 dm(y)\r) \\
	& g(0, \t{m})^{2} \leq K \l(1 + \int_{\R^{n}} y^2 d\t{m}(y)\r) 
\end{aligned}
\end{equation*}
\end{assumpB}
\begin{assumpB}\label{lip_m_given}
We have assumptions on both Lipschitz property and monotonicity.
\begin{enumerate}[(a)]
\item
\label{lip_m_given_a}
For $\phi= b, \sigma,\tsg,f$ and $m, m'\in\mathscr{P}_{2}(\R^{2n})$
\begin{align*}
	&| \phi(t, \theta,m) - \phi(t, \theta',m') |^{2}\\
	&\quad \leq K (|\Delta x|^{2}+ |c^{(1)}_t(\Delta y)+c^{(2)}_t(\Delta z) + c^{(3)}_t(\Delta \tz')|^{2} + \W^2_2(m,m'))
\end{align*}
where $\theta=(x,y,z,\tz)$, $\theta'=(x',y',z',\tz') \in \R^p$. We also assume, for $x,x' \in \R^n$, $\t{m},\t{m}' \in \P_2(\R^n)$ ,
\begin{equation*}
	| g(x,\t{m})-g(x',\t{m}') |^2 \leq K  \left(|\Delta x|^{2} + \W^2_2(\t{m},\t{m}')\right)
\end{equation*}

\item
\label{mon_m_given}
For any $m\in\mathscr{P}_{2}(\R^{2n}),\t{m} \in \P_2(\R^n)$, 
\begin{align*}
	&\la\D f_t, \D x \ra + \la \D b_t, \D y \ra + \la \D \sigma_t, \D z \ra +  \la \D \tilde{\sigma}_t, \D \tilde{z} \ra\\
	&\quad \leq  -\beta |c^{(1)}_t(\D y)+c^{(2)}_t(\D z) + c^{(3)}_t(\D \tz) |^2,
\end{align*}
for some $\beta > 0$ and
\begin{equation*}
	\mb{E}[\D g\D x ] \geq 0
\end{equation*}
where $\D f_t = f(t, x, y, z, \tz, m) - f(t, x', y', z', \tz', m)$, and $\D b_{t}$, $\D \sigma_{t}$, $\D \t\sigma_{t}$, $\D g$ are similarly defined.

\end{enumerate}
\end{assumpB}
 
Our main result for this section is the following

\begin{thm}\label{ue_existence} Assume \ref{measurable_mkv}-\ref{lip_m_given} hold and let $(X_t,Y_t,Z_t,\tZ_t)_{0 \leq t \leq T}$ denote the solution to FBSDE \rref{mkfbsde_eq}, then there exists a deterministic function $U: [0,T] \times \R^n \times \Ptwo(\R^n)$ such that
 \begin{equation}\label{decouple_smp} 
 	Y_t = U(t,X_t,\mb{P}_{X_t|\tF_t}), \quad \forall t \in [0,T] \quad \text{a.s.}
\end{equation}
\end{thm}

Theorem \ref{ue_existence} will be given as a consequence of Theorem \ref{markovian_main} presented at the end of this section. The main idea is to define explicitly 
a function $U$ through a solution of time-varying FBSDE with arbitrary initial data. This is done in two steps as it involves both the state variables and 
probability measure variable. Then, using \textit{a priori} estimates and a discretization argument, we show that this function $U$ satisfies \rref{decouple_smp} 
and, thus, is the decoupling function. The rest of the section are devoted to the proof of Theorem \ref{markovian_main}. We assume 
\ref{measurable_mkv}-\ref{lip_m_given} throughout the rest of this section.

\subsubsection{\Flow}\label{flowmap_subsec} In this section, we define $\flows$ $\{\Theta^{s,t}\}_{0 \leq s \leq t \leq T}$ which describes the conditional law at time $t$ of the solution of mean-field FBSDE over $[s,t]$ as a functional of the initial law. Our main result for this section, Theorem \ref{markovian_m}, gives the Markov property of the solution flows.   

We begin with its definition. For $m \in \Ptwo(\R^n)$, let $\xi \in \Ltwo_{\F_s}(\R^n)$ with law $m$ and denote by $\theta^{s,\xi}_{t} =(X^{s,\xi}_t,Y^{s,\xi}_t,Z^{s,\xi}_t,{\tZ}^{s,\xi}_t)_{s \leq t \leq T}$ the unique solution to FBSDE $(s,\xi,\{\tF^s_t\}_{s \leq t \leq T})$, i.e. it satisfies 
\begin{equation}
	\label{fbsde_st}
	\begin{aligned}
		dX^{s,\xi}_{t} &= b(t, \theta^{s,\xi}_{t}, \mb{P}_{(X^{s, \xi}_t,Y^{s, \xi}_t)|\tF_t^s})dt 
		+ \sigma(t, \theta^{s,\xi}_{t}, \mb{P}_{(X^{s, \xi}_t,Y^{s, \xi}_t)|\tF_t^s})dW_{t}\\
		&\quad + \tsg(t, \theta^{s,\xi}_{t}, \mb{P}_{(X^{s, \xi}_t,Y^{s, \xi}_t)|\tF_t^s})d\tilde{W}_{t} \\
		dY^{s,\xi}_{t} &= f(t, \theta^{s,\xi}_{t}, \mb{P}_{(X^{s, \xi}_t,Y^{s, \xi}_t)|\tF_t^s})dt + Z^{s,\xi}_{t}dW_{t} + \tZ^{s,\xi}_{t}d\tilde{W}_{t}\\
		X^{s,\xi}_{s} &= \xi, \quad Y^{s,\xi}_{T} = g(X^{s,\xi}_{T}, \mb{P}_{X^{s, \xi}_T|\tF_T^s}).
	\end{aligned}
\end{equation}
Recall that $\tF^s_t$ is a $\sigma$-algebra generated by the common Brownian motion starting at time $s$, i.e. $\tF^s_t = \sigma\l(\tW_r-\tW_s; s \leq r \leq t \r)$. We define, for $0 \leq  s\leq t \leq T$, the following two \textit{\flows} $\Theta^{s,t}: \Ptwo(\R^n) \to \Ltwo_{\F^s_t}(\Ptwo(\R^{2n}))$, $\Theta_X^{s,t}:\Ptwo(\R^n) \to \Ltwo_{\F^s_t}(\Ptwo(\R^n))$ as 
\begin{equation}\label{theta}
\begin{aligned}
	\Theta^{s,t}(m)\triangleq \mb{P}_{(X^{s, \xi}_t,Y^{s, \xi}_t)|\tF_t^s}, \quad \Theta_X^{s,t}(m)\triangleq \mb{P}_{X^{s, \xi}_t|\tF_t^s}
\end{aligned}
\end{equation}
We will sometimes use the following notation
\begin{equation}\label{msmt} 
	m^{s,m}_t \triangleq  \Theta^{s,t}(m), \quad m^{s,m}_{X,t} \triangleq  \Theta_X^{s,t}(m)
\end{equation} 

First, we check that this map is well-defined. That is, the conditional law $\mb{P}_{X^{s,\xi} | \tF^s_t}$ is independent of the choice of $\xi \in 
\Ltwo_{\F_s}(\R^n)$ provided that $\mb{P}_{\xi} = m$. This is equivalent to a (conditional) weak uniqueness for FBSDE, or equivalently, the Yamada-Watanabe theorem, 
extended to mean-field FBSDE with conditional law. We state a slightly more general result taking into account the conditional law as it will be applied in a 
subsequent section.

\begin{prop}\label{weakU} Let $0 \leq r \leq s \leq T$. Suppose that $\xi_1,\xi_2 \in \Ltwo_{\F_s}(\R^n)$ such that $\mb{P}_{\xi_1| \tF^r_s}=\mb{P}_{\xi_2| \tF^r_s} \in \Mt{\F_s}$, then $\mb{P}_{(X^{s, \xi_{1}}_t,Y^{s, \xi_{1}}_t)|\tF_t^r} = \mb{P}_{(X^{s, \xi_{2}}_t,Y^{s, \xi_{2}}_t)|\tF_t^r}$, and in particular $\mb{P}_{X^{s,\xi_1}_t| \tF^r_t} = \mb{P}_{X^{s,\xi_2}_t| \tF^r_t}$ for all $t \in [s,T]$ where $(X^{s,\xi_1}_t,Y^{s,\xi_1}_t),(X^{s,\xi_2}_t,Y^{s,\xi_2}_t)$ are as defined above.
\end{prop}

\begin{proof} Fix a path of the common Brownian motion $\tw \in \t{\Omega}$, then follow the same argument as in Theorem 5.1 in \cite{antonelli2003weak} which shows that pathwise uniqueness implies weak uniqueness for an FBSDE.
\end{proof}

Next, we gives a Lipschitz bound on this map.
 
\begin{prop}\label{est_theta_prop}
For $0 \leq s \leq t \leq T$, $m,m' \in \Ptwo(\R^n)$, there exists a constant $C_{K,T}$ that  depends only on $K,T$ such that
\begin{equation}
	\label{est_theta} 
	\begin{aligned}
	\mb{E}\l[ \W^2_2(\Theta^{s,t}(m),\Theta^{s,t}(m')) \r] &\leq C_{K,T}\W^2_2(m,m') \\
	\mb{E}\l[ \W^2_2(\Theta_X^{s,t}(m),\Theta_X^{s,t}(m')) \r] &\leq C_{K,T}\W^2_2(m,m')
	\end{aligned}
\end{equation}
\end{prop}
\begin{proof} Let $\xi$, $\xi'$ be arbitrary elements of $\Ltwo_{\F_s}(\R^n)$ with law $m$, $m' \in \P_2(\R^n)$. Let $(X_t,Y_t,Z_t,\tZ_t)_{0 \leq t \leq T}$ and $(X'_t,Y'_t,Z'_t,\tZ'_t)_{0 \leq t \leq T}$ denote the solutions of FBSDE \rref{fbsde_st} with initial $X_s=\xi,X'_s=\xi'$, then by the estimate \rref{estimate_diff_mkv}, it follows that
\begin{equation*}
\begin{aligned}
	\mb{E}[\W^{2}_2(\Theta^{s,t}(m),\Theta^{s,t}(m'))] &\leq \mb{E}[(X_t-X'_t)^2] + \mb{E}[(Y_t-Y'_t)^2] \\ 
	&\leq C_{K,T}\mb{E}[(\xi-\xi')^2]
\end{aligned}
\end{equation*}
for a constant $C_{K,T}$ depends only on $K,T$. Since $\xi$, $\xi'$ are arbitrary, we conclude that
\begin{equation*}
	\mb{E}[\W^2_2(\Theta^{s,t}(m),\Theta^{s,t}(m'))] \leq C_{K,T}\W^2_2(m,m')
\end{equation*}
The proof for $\Theta_X^{s,t}$ is identical.
\end{proof}

We are now ready to state and prove our main result for this subsection which gives Markov property of the law of the solution to the conditional mean-field FBSDE \rref{fbsde_st}

\begin{thm}\label{markovian_m} For any $m \in \Ptwo(\R^n)$ and $0 \leq  s \leq t \leq u \leq T$
\begin{equation}\label{consistent}  
\begin{aligned}
\Theta^{t,u}(\Theta_X^{s,t}(m))&=\Theta^{s,u}(m)\\
\Theta_X^{t,u}(\Theta_X^{s,t}(m))&=\Theta_X^{s,u}(m)
\end{aligned}
\end{equation}
\end{thm}

\begin{proof}  Let $\eta \in \Ltwo_{\F_s}(\R^n)$ with $\mb{P}_{\eta}= m$ and $(X^{s,\eta}_t,Y^{s,\eta}_t,Z^{s,\eta}_t,\tZ^{s,\eta}_t)_{s \leq t \leq T}$ denote the solution to FBSDE \rref{fbsde_st} corresponding to the definition of $\Theta^{s,u}$ and $\Theta_{X}^{s,u}$, so
\begin{equation*}
	\Theta^{s,u}(m) = \mb{P}_{(X^{s,\eta}_u, Y^{s,\eta}_u)  | \tF^s_u}, \qquad \Theta_{X}^{s,u}(m) = \mb{P}_{X^{s,\eta}_u | \tF^s_u}
\end{equation*}
Since $\Ptwo(\R^{n})$ is separable, for any $\delta > 0$, there exist a sequence of disjoint Borel measurable subsets $\{ A_{n} \}_{n \in \mb{N}}$ of $\Ptwo(\R^{n})$ such that $\text{diam}(A_{n}) < \delta$ and $\cup_{n \in \mb{N}} A_n = \Ptwo(\R^{n})$. Let $m_n$ be a representative element of $A_n$ so that $\W_2(m,m_n) < \delta$ for all $m \in A_n$. Let $B_n = \{ \omega \in \Omega; \mb{P}_{X^{s,\eta}_t | \tF^s_t}(\omega) \in A_n \}$. Consider
\begin{equation*}
	\t{\xi} \triangleq \sum_{n \in \mb{N}} \Ind_{B_n}\xi^n
\end{equation*}
where $\xi^n \in \Ltwo_{\F_t}$ has law $m_n$ and is independent of $\tF_t$, thus independent of $B_n$. That is, 
\begin{equation*}
	\mb{P}_{\xi^n | \tF^s_t} = \mb{P}_{\xi^n} = m_n
\end{equation*}
Then it follows by construction that $ \t{\xi}\in \Ltwo_{\F_t}(\R^n)$ and
\begin{equation*}
	\W_2(\mb{P}_{\t{\xi} | \tF^s_t}, \mb{P}_{X^{s,\eta}_t | \tF^s_t}) < \delta
\end{equation*}
Using this type of discretization and Lemma \ref{representation_m_xi}, we can redivide $A_n$ further and proceed sequentially to construct a sequence $\{\xi^N\}_{N \in \mb{N}}$ of the form
\begin{equation*}
	\xi^N \triangleq \sum_{n \in \mb{N}} \Ind_{B_{n,N}}\xi^{n,N}
\end{equation*}
such that $\{\xi^N\}_{N \in\mb{N}}$ is Cauchy in $\Ltwo_{\F_t}$, $\xi^{n,N}$ is independent of $\tF^s_t$ and 
\begin{equation*}
	\W_2(\mb{P}_{\xi^N | \tF^s_t}, \mb{P}_{X^{s,\eta}_t | \tF^s_t}) < \frac{1}{N}
\end{equation*}
Let $ \xi = \lim_{N \to \infty} \xi^N $ in $\Ltwo_{\F_t}$, then we have 
\begin{equation}\label{eq1}
	\mb{P}_{\xi | \tF^s_t} = \mb{P}_{X^{s,\eta}_t | \tF^s_t} = \Theta_{X}^{s,t}(m)
\end{equation}
Now consider FBSDE $(t,\xi,\{\tF^s_r\}_{t \leq r \leq T})$ and denote its solution by $(X^{t,\xi}_r,Y^{t,\xi}_r,Z^{t,\xi}_r,\tZ^{t,\xi}_r)$. By \rref{eq1} and Theorem \ref{weakU}, it follows that
\begin{equation}\label{eq2}
	 \mb{P}_{(X^{t,\xi}_u, Y^{t,\xi}_u )| \tF^s_u} = \mb{P}_{(X^{s,\eta}_u, Y^{s,\eta}_u ) | \tF^s_u} = \Theta^{s,u}(m)
\end{equation}
Let
\begin{equation*}
	X^{N}_r \triangleq \sum_{n \in \mb{N}} \Ind_{B_{n,N}}X^{n,N}_r;\quad Y^{N}_r \triangleq \sum_{n \in \mb{N}} \Ind_{B_{n,N}}Y^{n,N}_r
\end{equation*}
where $(X^{n,N}_r, Y^{n,N}_r)_{t \leq r \leq u}$ is a solution to FBSDE $(t,\xi^{n,N},\{\tF^t_r\}_{t \leq r \leq T})$. It is easy to check that $(X^{N}_r, Y^{N}_r)_{t \leq r \leq u}$ is a solution to FBSDE $(t,\xi^{N},\{\tF^s_r\}_{t \leq r \leq T})$  with initial $\xi^N$. Note that $(X^{n,N}_r, Y^{n,N}_r)$ is $\F^t_r$-measurable which is independent of $\tF_t$, hence independent of $B_n$. Thus, we have
\begin{equation*}
	\mb{P}_{(X^N_u, Y^N_u) | \tF^s_u} = \sum_{n\in \mb{N}} \Ind_{B_{n,N}} \mb{P}_{(X^{n,N}_u, Y^{n,N}_u) | \tF^t_u} = \sum_{n\in \mb{N}} \Ind_{B_{n,N}} \Theta^{t,u}(\mb{P}_{\xi^{n,N}})
\end{equation*}
Taking limit in $\Ltwo_{\F_u}$ as $N \to \infty$ both sides, it follows from the fact that $\mb{E}[(\xi - \xi^N)^2] \to 0$ and from estimate \rref{estimate_diff_mkv} that
\begin{equation*}
	\mb{P}_{(X^{t,\xi}_u, Y^{t,\xi}_u) | \tF^s_u} = \Theta^{t,u}(\mb{P}_{\xi | \tF^s_t})
\end{equation*}
Combine with \rref{eq1} and \rref{eq2}, we get \rref{consistent} as desired. With similar proof, we also have
\begin{equation*}
	\Theta_{X}^{t,u}(\Theta_X^{s,t}(m))=\Theta_{X}^{s,u}(m)
\end{equation*}
\end{proof}


\subsubsection{Defining a decoupling function}  Now, we let $\xi \in \Ltwo_{\F_s}(\R^n)$ and define $(X^{s,\xi,m}_t,Y^{s,\xi,m}_t,Z^{s,\xi,m}_t,{\tZ}^{s,\xi,m}_t)_{s \leq t \leq T}$ to be the $\F_t$-adapted solution to the following FBSDE
\begin{equation}\label{mkfbsde_sub}
	\begin{aligned}
& dX^{s,\xi,m}_t = b(t, \theta^{s, \xi, m}_{t}, m^{s,m}_t) dt + \sigma(t, \theta^{s, \xi, m}_{t}, m^{s,m}_t) dW_{t} + \tilde{\sigma}(t, \theta^{s, \xi, m}_{t}, m^{s,m}_t)  \tdW_{t} \\
 &dY^{s,\xi,m}_t  =f(t, \theta^{s, \xi, m}_{t}, m^{s,m}_t) dt + Z^{s,\xi,m}_t dW_{t} + \tZ^{s,\xi,m}_t\tdW_{t} \\
& X^{s,\xi,m}_s = \xi, \quad Y^{s,\xi,m}_T = g(X^{s,\xi,m}_T,m^{s,m}_T) 
	\end{aligned}
\end{equation}

\begin{remark} 
\begin{enumerate}
\item The initial $\xi \in \Ltwo_{\F_s}(\R^n)$ does not necessarily have law $m$. Here, $m \in \Ptwo(\R^n)$ and hence $(m^{s,m}_t)_{s \leq t \leq T}$ are given exogenously.
\item The FBSDE \rref{mkfbsde_sub} is a classical FBSDE with random coefficients and not a mean-field FBSDE since the stochastic law $(m^{s,m}_t)_{s \leq t \leq T}$ in the system is given exogenously.
\item The law $m$ in $m^{s,m}_t$ refers to the law of $X_t$ which is an element in $\Ptwo(\R^n)$ while $\Theta^{s,t}(m)$ or $m^{s,m}_t$ refers to the joint law of $(X_t,Y_t)$, an element in $\Ptwo(\R^{2n})$. 
\end{enumerate}
\end{remark}

The assumptions \ref{linearGrowthOnM_m_given}-\ref{lip_m_given} ensure the existence and uniqueness of the FBSDE above using Theorem \ref{wellposed} similar to our proof for Theorem \ref{wellposed_mckean} with different lifting functionals.

\begin{thm}\label{mkfbsde_sub_wellposed} Assume that \ref{measurable_mkv}-\ref{lip_m_given} hold. The FBSDE in \rref{mkfbsde_sub} has a unique solution $(X_{t},Y_{t},Z_{t},\tZ_{t})_{s \leq t \leq T}$. 
\end{thm}
\begin{proof}
We need to verify that the system in \rref{mkfbsde_sub} with \ref{measurable_mkv}-\ref{lip_m_given} satisfies the assumptions \ref{measurable_bfg}-\ref{lip_mon_bfg} where
\begin{equation}
\begin{aligned}
&B(t, X_{t}, Y_{t}, Z_{t}, \tZ_{t}) = b(t, X_{t}, Y_{t}, Z_{t}, \tZ_{t}, m^{s,m}_t)\\
&F(t, X_{t}, Y_{t}, Z_{t}, \tZ_{t}) = f(t, X_{t}, Y_{t}, Z_{t}, \tZ_{t}, m^{s,m}_t)\\
&\Sigma(t, X_{t}, Y_{t}, Z_{t}, \tZ_{t}) = (\sigma(t, X_{t}, Y_{t}, Z_{t}, \tZ_{t}, m^{s,m}_t), \t\sigma(t, X_{t}, Y_{t}, Z_{t}, \tZ_{t}, m^{s,m}_t))
\end{aligned}
\end{equation}
The result then follows from Theorem \ref{wellposed}. Assumption \ref{measurable_bfg} follows from the fact that $b,f,\sigma,\tsg$ is measurable. For \ref{bd_bfg}, by using \ref{linearGrowthOnM_m_given} and \rref{estimate_mkv}, we have
\begin{equation*}
\begin{split}
\mb{E}\int^{T}_{s}|B(t,0,0,0,0)|^{2}dt &=\mb{E}\int^{T}_{s}|b(t,0,0,0,0, m^{s,m}_t)|^{2}dt\\
& \leq \mb{E} \int^{T}_{s}\int_{\R^{2n}} |x|^{2}dm^{s,m}_t(x) dt\\
&=\mb{E} \int^{T}_{s}\mb{E}(|X_{t}^{s, \bar{\xi}}|^{2} + |Y_{t}^{s, \bar{\xi}}|^{2} | \tF^s_t )dt\\
&=\int^{T}_{s}\mb{E}(|X_{t}^{s, \bar{\xi}}|^{2} + |Y_{t}^{s, \bar{\xi}}|^{2} )dt\\
&\leq T(\mb{E}\sup_{s\leq t\leq T}|X_{t}^{s, \bar{\xi}}|^{2} + \mb{E}\sup_{s\leq t\leq T}|Y_{t}^{s, \bar{\xi}}|^{2}   )\\
&<\infty
\end{split}
\end{equation*}
where $(X_{t}^{s, \bar{\xi}}, Y_{t}^{s, \bar{\xi}})$ solves the FBSDE in \rref{fbsde_st} with initial $\bar{\xi} \in \Ltwo_{\F_s}(\R^n)$ such that $\mb{P}_{\bar{\xi}}=m$. Lastly, from \ref{lip_mon_mkv}, the condition \ref{lip_mon_bfg} holds pointwise and thus holds under conditional expectation.

\end{proof}

Note that the initial $\xi$ in \rref{mkfbsde_sub} is arbitrary and does not necessarily have law $m$. When $\xi=x$ is a constant, $Y^{s,x,m}_s$ is deterministic since it is $\F_s^s$-measurable. This fact allows us to define the following map 
\begin{equation}\label{def_ue}
\begin{aligned}
	U: &[0,T] \times \R^n \times \Ptwo(\R^n) \to \R^{n} \\
				&(s,x,m)\mapsto Y^{s,x,m}_s
\end{aligned}
\end{equation}

To summarize how we define $U(s,x,m)$. We begin with the law $m \in \Ptwo(\R^n)$, then solve the \textit{mean-field FBSDE in conditional law} over $[s,T]$ with law $m$ as initial to get the stochastic flow of probability measure $(m^{s,m}_t)_{s \leq t \leq T}$. Then we solve \eqref{mkfbsde_sub} with $(m^{s,m}_t)_{s \leq t \leq T}$ given exogenously and initial $X_s=x$ which is \textit{a classical FBSDE with random coefficients}.   

We will show in Theorem \ref{markovian_main} that this map is indeed our decoupling function. We begin with estimates of the related FBSDEs.  

\begin{prop}\label{diff_ind_prop}
Assume \ref{measurable_mkv}-\ref{lip_m_given}. For $i=1,2$, let $\xi^{(i)} \in \Ltwo_{\F_s}(\R^n)$, $(m^{(i)}_t)_{s \leq t \leq T} \in \MP{s}{T}$, and $(X^{(i)}_t,Y^{(i)}_t,Z^{(i)}_t,\tZ^{(i)}_t)_{s \leq t \leq T}$ denote the solution to FBSDE \rref{mkfbsde_sub} given $m^{(i)}$ and initial $\xi^{(i)}$, then the following estimate holds
\begin{equation}\label{estimate_diff_ind_m}
\begin{aligned}
	&\mb{E}\l[ \sup_{s \leq t \leq T} \Ind_A |\Delta X_t|^2 + \sup_{s \leq t \leq T} \Ind_A |\Delta Y_t|^2  +  \int_s^T \Ind_A [|\Delta Z_t|^2 + |\Delta \tZ_t|^2] dt \r] \\
	&\qquad \leq C_{K,T}\mb{E}[\Ind_A |\Delta\xi|^2 + \Ind_A \int_s^T (\Delta m_t)^2 dt + \Ind_A (\D m^{n}_T)^2 ]
\end{aligned}
\end{equation}
where $m^{n,(i)}_T$ denotes the marginal distribution of $m^{(i)}_T$ in the first $n$ dimension, $\D X_t = X^{(1)}_t-X^{(2)}_t$, $\D m_t = \W_2(m_t^{(1)},m_t^{(2)})$, and $\D Y_t, \D Z_t, \D \tZ_t, \D \xi, \D m^{n}_T$ are defined similarly. 
\end{prop} 
\begin{proof}  Let $(B^{(i)},\Sigma^{(i)},F^{(i)},G^{(i)})$, for $i=1,2$, be the functionals defined as
\begin{equation*}
\begin{aligned}
&B^{(i)}(t, X_{t}, Y_{t}, Z_{t}, \tZ_{t}) = b(t, X_{t}, Y_{t}, Z_{t}, \tZ_{t}, m^{(i)}_t)\\
&F^{(i)}(t, X_{t}, Y_{t}, Z_{t}, \tZ_{t}) = f(t, X_{t}, Y_{t}, Z_{t}, \tZ_{t}, m^{(i)}_t)\\
&\Sigma^{(i)}(t, X_{t}, Y_{t}, Z_{t}, \tZ_{t}) = (\sigma(t, X_{t}, Y_{t}, Z_{t}, \tZ_{t}, m^{(i)}_t), \t\sigma(t, X_{t}, Y_{t}, Z_{t}, \tZ_{t}, m^{(i)}_t)) \\
&G^{(i)}(X_T) = g(X_{T}, m^{n,(i)}_{T})
\end{aligned}
\end{equation*}
Then as shown in Theorem \ref{mkfbsde_sub_wellposed}, $(B^{(i)},\Sigma^{(i)},F^{(i)},G^{(i)})$ satisfies \ref{measurable_bfg}-\ref{lip_mon_bfg}. Thus, by estimate \rref{estimate_diff_bfg}, we have
\be\label{bfg_diff}
\begin{aligned}
&\mb{E} \l[ \sup_{s\leq t \leq T} \Ind_{A}|\Delta X_{t}|^{2} + \sup_{s\leq t \leq T} \Ind_{A}|\Delta Y_{t}|^{2} + \int^{T}_{s}\Ind_{A}(|\Delta Z_{t}|^{2}+|\Delta \tZ_{t}|^{2})dt \r]\\ 
&\quad \leq C_{K, T}\mb{E}\l[ \Ind_{A} \l( |\Delta \xi|^{2} +|\bar{G}|^{2} + \int^{T}_{s}(|\bar{F}_{t}|^{2} + |\bar{B}_{t}|^{2} + |\bar{\Sigma}_{t}|^{2} \r) dt \r] 
\end{aligned}
\ee
where $\b{B_t} = (B^{(1)}-B^{(2)})(t,X^{(1)}_t,Y^{(1)}_t,Z^{(1)}_t, \tZ^{(1)}_t)$ and similarly for other terms. Note that by \ref{lip_m_given}\rref{lip_m_given_a}
\begin{equation}\label{bfg_diff_2}
\begin{aligned}
 |\b{B_t}| &= |(B^{(1)}-B^{(2)})(t,X^{(1)}_t,Y^{(1)}_t,Z^{(1)}_t, \tZ^{(1)}_t)| \\
 & = | b(t, X^{(1)}_t,Y^{(1)}_t,Z^{(1)}_t, \tZ^{(1)}_t, m^{(1)}_t) - b(t, X^{(1)}_t,Y^{(1)}_t,Z^{(1)}_t, \tZ^{(1)}_t, m^{(2)}_t)| \\
 & \leq \W_2(m^{(1)}_t,m^{(2)}_t)
 \end{aligned}
 \end{equation}
and similarly for other terms. The estimate \rref{estimate_diff_ind_m} then follows from \rref{bfg_diff} and \rref{bfg_diff_2}.
\end{proof}

To complete the proof of the Markov property of FBSDE \rref{fbsde_gen_2}, we are left to show \rref{decouple_smp}. We first state necessary estimates for $U$. 

\begin{lemma}\label{estimate_ue} Let $U:[0,T] \times \R^n \times \Ptwo(\R^n) \to \R^{n}$ be as defined above, then it satisfies
	\begin{align} &| U(t,x,m) - U(t',x',m')|  \label{convexUe2}\\
	&\qquad \leq C_{K,T} \l(  | x-x'| + \W_2(m,m')+\l(1+|x|+\l(\int_{\R^n}y^2 dm(y)\r)^{\frac{1}{2}} \r)\sqrt{|t-t'|} \r)\nonumber \\
	&\l( U(t,x,m)-U(t,x',m)\r)(x-x') \geq 0 
 \end{align} 
for all $t \in [0,T], x,x' \in \R, m,m' \in \Ptwo(\R^n)$, where $C_{K,T}$ depends only on $K,T$.
\end{lemma}

\begin{proof} From Proposition \ref{est_theta_prop} and Proposition \ref{diff_ind_prop} , we get
\begin{equation}\label{est_u_xm}
 | U(t,x,m)-U(t,x',m') | \leq C_{K,T} \l( | x-x'| + \W_2(m,m') \r)
 \end{equation}
Next, by definition of $U(t,x,m), U(t',x,m)$, we need to consider FBSDE over different time and filtration. We assume $t' \geq t$ and let $(X^{t,x,m}_u$, $Y^{t,x,m}_u$, $Z^{t,x,m}_u$, $\tZ^{t,x,m}_u)_{t \leq u \leq T}$, $(X^{t',x,m}_u,Y^{t',x,m}_u,Z^{t',x,m}_u,\tZ^{t',x,m}_u)_{t' \leq u \leq T}$ denote the solutions to FBSDE \rref{mkfbsde_sub} corresponding to the definition of $U(t,x,m)$ and $U(t',x,m)$ respectively. We can extend the latter to $[t,T]$ by setting the coefficients to $0$ for $s \in [t,t')$ which still satisfy the same assumptions. Thus, by Theorem \ref{estimate_fbsde}, Theorem \ref{markovian_m}, Proposition \ref{est_theta_prop}, we have

\be\label{bd_um_1}
\begin{aligned}
&\mb{E}[\sup_{t' \leq u \leq T} (Y^{t',x,m}_u-Y^{t,x,m}_u)^2  ] \\
&\quad \leq C_{K,T} \Big( \int_t^{t'} \mb{E}\l[ (X^{t,x,m}_u)^2+(Y^{t,x,m}_u)^2+(Z^{t,x,m}_u)^2+(\tZ^{t,x,m}_u)^2 \r] du  \\
&\quad\quad  +  \int_{t}^T \W^2_2(\Theta^{t,u}(m),\Theta^{t',u}(m))du   \Big) \\
& \quad \leq C_{K,T}\l( (1+x^2)(t'-t) + \int_{t}^T \mb{E}\l[ \W^2_2(\Theta^{t',u}(\Theta_{X}^{t,t'}(m)),\Theta^{t',u}(m)) \r] du \r) \\
& \quad \leq C_{K,T}\l( (1+x^2)(t'-t) + \int_{t}^T \mb{E}\l[ \W^2_2(\Theta_X^{t,t'}(m),m) \r] du \r) 
\end{aligned}
\ee
 
where $C_{K,T}  > 0$ is a constant which may differ from line to line. Now, consider the mean-field FBSDE $(t,\xi,\{\tF^t_u\}_{t \leq u \leq T})$ with $\mb{P}_\xi=m$ and the same FBSDE but with functional
$$ \Phi(s,X,Y,Z,Z')=\begin{cases} 0 &,t \leq s\leq t' \\ \phi(s,X,Y,Z,Z',\mb{P}_{(X,Y)|\tF^{t'}_s}) &,s \geq t' \end{cases}$$  
for $\Phi=B,\Sigma,F$ and $\phi=b,\sigma,f$ respectively. Denote their solutions by

\begin{equation*}
	(X^{t,m}_u,Y^{t,m}_u,Z^{t,m}_u,{\tZ}^{t,m}_u)_{t\leq u \leq T},\quad
    (X^{t',m}_u,Y^{t',m}_u,Z^{t',m}_u,{\tZ}^{t',m}_u)_{t\leq u \leq T},
\end{equation*}
respectively.  Thus, by Theorem \ref{estimate_fbsde}, assumption \ref{linearGrowthOnM_m_given}, and Proposition \ref{est_theta_prop}, we have

\be\label{bd_um_2}
\begin{aligned}
\mb{E}[(X^{t',m}_u-X^{t,m}_u)^2  ] &\leq C_{K,T} \l( \int_t^{t'} \mb{E}\l[ (X^{t,m}_u)^2+(Y^{t,m}_u)^2+(Z^{t,m}_u)^2+(\tZ^{t,m}_u)^2  \r] du \r)  \\
& \leq C_{K,T}\l( (1+\mb{E}|\xi|^2) \r) (t'-t) 
\end{aligned}
\ee
Therefore,

\be\label{bd_um_3}
\begin{aligned}
\mb{E}\l[ \W^2_2(\Theta_{X}^{t,t'}(m),m) \r] &\leq \mb{E}[(X^{t,m}_{t'}-\xi)^2 ] \\
&  \leq  \mb{E}[(X^{t,m}_{t'}-X^{t',m}_{t'})^2  ] \\
& \leq C_{K,T}\l( (1+\mb{E}|\xi|^2)(t'-t) \r)  \\
& = C_{K,T} \l( 1 + \int_{\R^n} |y|^2 dm(y) \r) (t-t')
\end{aligned}
\ee
Combining \rref{est_u_xm}, \rref{bd_um_1}, and \rref{bd_um_3} yields \eqref{convexUe2} as desired.

Lastly, let $(X_t,Y_t,Z_t,\tZ_t)_{s \leq t \leq T}$ and $(X'_t,Y'_t,Z'_t,\tZ'_t)_{s \leq t \leq T}$ denote the solutions to the FBSDE corresponding to the definition of $U(s,x,m)$ and $U(s,x',m)$ respectively. Note that both FBSDE has the same coefficient functions and only the initials are different. Let $\D X_t = X_t - X'_t$ and define similarly $\D Y_t$, $\D Z_t$, $\D \tZ_t$, $\D b_t$, $\D f_t$, $\D \sigma_t$,  $\D \t\sigma_t$, $\D g$. Applying \ito's lemma to $\la \D X_t,\D Y_t \ra$ and using \ref{lip_m_given}\rref{mon_m_given} yields

\begin{align*}
&\mb{E}\la \D Y_s,\D X_s \ra = \mb{E}\la \D g, \D X_T \ra\\ 
&\quad -\mb{E} \int_s^T \l( \la \D f_{t}, \D X_t \ra + \la \D b_{t}, \D Y_t \ra + \la \D \sigma_{t} ,\D Z_{t} \ra  + \la \D \t\sigma_{t} \D \tZ_{t} \ra \r)  dt
\geq 0
\end{align*}
By definition of $U$ and the fact that it is deterministic, we deduce that

\begin{equation*}
	(U(s,x,m)-U(s,x',m))(x-x') \geq 0
\end{equation*}

\end{proof}

Now we are ready to state and prove the existence of a deterministic decoupling function thereby establishing the Markov result. Using Theorem \ref{markovian_m} above, we can show \rref{decouple_smp} using a similar argument as was done for a classical FBSDE (see Corollary 1.5 in \cite{delarue2002existence} for instance).

\begin{thm} \label{markovian_main}  Let $s \in [0,T], m \in \Ptwo(\R^n), \xi \in \Ltwo_{\F_s}(\R^n)$, consider $(m^{s,m}_t)_{s \leq t \leq T}$ and $(X^{s,\xi,m}_t,Y^{s,\xi,m}_t,Z^{s,\xi,m}_t,{\tZ}^{s,\xi,m}_t)_{s \leq t \leq T}$ as defined above, then it follows that
\begin{equation}\label{decouple_smp_2}
Y^{s,\xi,m}_t = U(t, X^{s,\xi,m}_t, m^{s,m}_{X,t}), \quad \forall t \in [s,T] \text{ a.s.}
\end{equation}
\end{thm}

\begin{remark} \rref{decouple_smp} in Theorem \ref{ue_existence} follows from \rref{decouple_smp_2} by setting $s=0$ and $\mb{P}_{\xi}=m$.
\end{remark}
\begin{proof}

 We will use a similar argument as in the proof of Theorem \ref{markovian_m} which is based on a discretization argument and global Lipschitz property. Note that $\R^{n} \times \Ptwo(\R^n)$ is separable, hence there exists a countable disjoint set $\{ A_n \}_{n \in \mb{N}}$ such that $\bigcup_{n=1}^\infty A_n = \R^{n} \times \Ptwo(\R^n)$ and $\text{diam}(A_n) < \delta$. Let $(x_n,m_n) \in \R^{n} \times \Ptwo(\R^n)$ be a fixed element of $A_n$, then let 
\begin{equation}\label{Bn}
	B_n = \{ \omega \in \Omega ; (X^{s,\xi,m}_t, m^{s,m}_{X,t}) \in A_n \}
\end{equation}
Then by Lemma \ref{estimate_ue}, we have
\begin{equation}\label{Ue1}
\sum_{n \in \mb{N}}| U(t,X^{s,\xi,m}_t,m^{s,m}_{X,t})-U(t,x_n,m_n) |^2 \Ind_{B_n}  \leq C_1\delta^2
\end{equation}
On the other hands, using Theorem \ref{markovian_m}, it follows that $(X^{s,\xi,m}_r,Y^{s,\xi,m}_r,Z^{s,\xi,m}_r,\tZ^{s,\xi,m}_r)_{t \leq r \leq T}$ satisfies the FBSDE

\begin{equation*}\label{fbsde3}
\begin{aligned}
	dX_r &= b(r,X_r,Y_r,Z_r,\tZ_r,\Theta^{t,r}(m^{s,m}_t))dr + \sigma(r,X_r,Y_r,Z_r,\tZ_r,\Theta^{t,r}(m^{s,m}_t)) dW_r \\
	&\qquad+ \tsg(r,X_r,Y_r,Z_r,\tZ_r,\Theta^{t,r}(m^{s,m}_t)) d\tW_r \\
	dY_r &= f(r,X_r,Y_r,Z_r,\tZ_r,\Theta^{t,r}(m^{s,m}_t))dt+Z_r dW_r + \tZ_r d\tW_r \\
	X_t &= X^{s,\xi,m}_t, \quad Y_T = g(X_T, \Theta_X^{t,T}(m^{s,m}_t))
\end{aligned}
\end{equation*}
Thus, we get by Proposition \ref{est_theta_prop}, Proposition \ref{diff_ind_prop}, and \rref{Bn} that
\begin{equation}\label{Y1}
\sum_{n \in \mb{N}}\mb{E}\l[ (Y^{s,\xi,m}_t - Y^{t,x_n,m_n}_t)^2\Ind_{B_n} \r] \leq  C_2\delta^2
\end{equation}
Combining \rref{Ue1} and \rref{Y1}, it follows that
\begin{equation*}
	\mb{E}\l[ \l( U(t,X^{s,\xi,m}_t,m^{s,m}_{X,t}) - Y^{s,\xi,m}_t \r)^2 \r] \leq C_3\delta^2
\end{equation*}
Since $\delta$ is arbitrary, \rref{decouple_smp_2} holds a.s. for each $t \in [0,T]$. Then by continuity in $t$ of $U$ and the fact that $(X^{t,\xi,m}_s,Y^{t,\xi,m}_s)_{t \leq s \leq T}$ have continuous trajectories, we have \rref{decouple_smp_2} as desired. 

\end{proof}


\section{Mean-field games with common noise}
\label{sec_mfg}


A mean-field game (MFG) is a system of differential equations to describe the evolution of the distribution of the players when each player maximizes its own utility 
and there are infinitely many players in the game.  The original framework are provided by Lasry and Lion \cite{cardaliaguet2010,Lasry2007} and its wellposedness 
are proved by the PDE approach.  Because of the nature of the problem, the probabilistic approach (for example, see \cite{carmona2013probabilistic}) quickly becomes 
a popular approach in the MFG community after Lasry and Lion's original work.  In the probabilistic approach, a mean-field game is modeled as a 
system of FBSDEs where the forward SDE describes the evolution of the system and the backward SDE determines the individuals' optimal control.  Because the system 
evolution and the optimal control affect each other, the forward and backward SDEs are fully coupled in general. 

The original MFG framework and largely the following literature assume that all the individuals' uncertainties/noises are independent; in other words, there is no 
common noise allowed in the system.  The independence assumption is required mainly due to the mathematical tractability; with a common noise, the PDE approach would 
lead to a system of forward-backward stochastic PDEs and many crucial techniques can not be applied in the presence of a common noise.  For the MFG with common 
noise, the probabilistic approach becomes a feasible method because with a common noise, the forward and backward SDEs would be coupled through the law of the 
solution conditional on the filtration of the common noise, and to provide the wellposedness result is still possible even in this case.  While there is a relatively 
small amount of the literature, the MFG with common noise has gained interest due to its applications in economics and financial modeling.  We refer readers to 
\cite{carmona2014commonnoise, Lacker2015, ahuja2015mean} for the theoretical analysis of MFGs with common noise and \cite{carmona2013mean} for the example of an 
application.


In this section, we consider a mean-field game (MFG) model in the presence of common noise. By applying results from section \ref{sec_mkvfbsde}, we establish existence and uniqueness of this class of models under linear-convex setting and weak monotone cost functions. In addition, we show that the solution to MFG with common noise is Markovian as a consequence of the existence of a decoupling function discussed in section \ref{subsec_decoupling}. 


\subsection{Problem Formulation}\label{subsec_mfg_formulation} Mean-field games (MFG) with common noise can be described in succinct form as follow; 
\begin{equation}\label{mfg_formulation}\begin{cases}
    \alpha^* \in \arg\max_{\alpha \in \Htwo{k}} \mb{E} \l[ \int_0^T f(t,X^\alpha_t,m_t,\alpha_t)dt + g(X^\alpha_T,m_T)\r]\\
    dX^\alpha_t = b(t,X^\alpha_t,m_t,\alpha_t)dt + \sigma(t,X^\alpha_t,m_t,\alpha_t)dW_t +\tsg(t,X^\alpha_t,m_t,\alpha_t)d\tW_t  \\
    m_t = \mb{P}_{X^{\alpha^*_t} | \tF_t}, \quad \tF_t = \sigma(\tW_s; 0 \leq s \leq t) 
  \end{cases}
\end{equation}
where the set up and notations are as defined in section \ref{sec_mkvfbsde} with the following measurable functions being given;
\begin{align*}
	&b: [0,T] \times \R^n \times \Ptwo(\R^n) \times \R^k \to \R^n,\quad \sigma: [0,T] \times \R^n \times \Ptwo(\R^n) \times \R^k \to \R^{n \times d_1} \\
	&\tsg : [0,T] \times \R^n \times \Ptwo(\R^n) \times \R^k \to \R^{n \times d_2},\quad f: [0,T] \times \R^n \times \Ptwo(\R^n) \times \R^k \to \R^n,\\
	&g: \R^n \times \Ptwo(\R^n) \to \R^n.
\end{align*}
To simplify the notations, we assume that $d_1=d_2=1$ although the result in this section still hold for any $d_1,d_2 > 0$. For convenience, we will refer to MFG with common noise ($\tsg \not\equiv 0$) as $\textit{\cMFG}$ and MFG without common noise ($\tsg \equiv 0$) as  $\textit{\oMFG}$ to emphasize the existence/non-existence of the common noise.
MFG is formulated as a heuristic limit of an $N$-player stochastic differential game: for $i=1,\ldots,N$
\begin{equation}\label{Nplayer_formulation}\begin{cases}
    \alpha^{i} \in \arg\max_{\alpha \in \Htwo{k}} \mb{E} \l[ \int_0^T f(t,X^{\alpha,i}_t,m^N_t,\alpha_t)dt + g(X^{\alpha,i}_T,m^N_T)\r],\\
    dX^{\alpha^i,i}_t = b(t,X^{\alpha^i,i}_t,m^N_t,\alpha^i_t)dt + \sigma(t,X^{\alpha^i,i}_t,m^N_t,\alpha^i_t)dW^i_t +\tsg(t,X^{\alpha^i,i}_t,m^N_t,\alpha^i_t)d\tW_t  \\
    m^N_t = \frac{1}{N} \sum_{i=1}^N \delta_{X^{\alpha^i,i}_t} 
  \end{cases}\end{equation}
where $\delta_a$ denotes the Dirac measure at $a \in \R^n$. We emphasize the main features of this $N$-player game which are essential to the formulation of MFG. First, the cost functions are identical for all other players as a function of his/her state, control, and other players' states. Second, the dependence on other players' states is only through the empirical distribution of all states, or equivalently, the interaction between players is only of a mean-field type. Lastly, the random noise in the players' state process consists of an independent component $W^i_t$ (individual noise) and a common random factor $\tW_t$ shared among all the players (common noise), all of which are mutually independent. 

Under these symmetric properties, the optimization problem is identical in the perspective of each players. Thus, when $N$ is large, we can replace the empirical distribution with the law of a single player and only consider a control problem of this \textit{representative player}. This single player optimization problem involving the law is precisely the MFG problem \rref{mfg_formulation}. It is important to note that this formulation of MFG via taking the limit as $N \to \infty$ is heuristic and the convergence or the relation between a solution to MFG and the finite player counterpart require non-trivial justifications. However, the topic is beyond the scope of this paper. We refer interested readers to \cite{bardi2014,fischer2014,gomes2012cont,lacker2014general}.

The MFG problem \rref{mfg_formulation} can also be viewed as a fixed point problem; Given a strategy $\bal \in \Htwo{k}$, we set $\bar{m}_t = \mb{P}_{X^{\bal}_t | \tF_t}$, then solve an \textit{individual control problem given $\bar{m}$};
\begin{equation}\label{ind_formulation}\begin{cases}
    \alpha^* \in \arg\max_{\alpha \in \Htwo{k}} \mb{E} \l[ \int_0^T f(t,X^\alpha_t,\bar{m}_t,\alpha_t)dt + g(X^\alpha_T,\bar{m}_T)\r]\\
    dX^\alpha_t = b(t,X^\alpha_t,\b{m}_t,\alpha_t)dt + \sigma(t,X^\alpha_t,\b{m}_t,\alpha_t)dW_t +\tsg(t,X^\alpha_t,\b{m}_t,\alpha_t)d\tW_t  \\
     \end{cases}\end{equation}
This step yields a new optimal control $\alpha^*$. It is clear from \rref{mfg_formulation} that the fixed point of this process gives the solution to MFG.

\subsection{Assumptions}\label{subsec_mfg_assumption} We now state the main assumptions on the model and cost functions. The first set of assumptions is essential for ensuring that given any stochastic flow of probability measure $m=(m_t)_{0\leq t \leq T} \in \MM{0}{T}$, the stochastic control for an individual player given $m$ is uniquely solvable. For notational convenience, we will use the same constant $K$ for all the conditions below.

\begin{assumpC}\label{linear}\label{mfg1_first_ass} The state process is linear in $(x,\alpha)$; for $\phi=b,\sigma,\tsg$, $\phi(t,x,m,\alpha) = \phi_0(t,m)+\la \phi_1(t,m),x \ra + \la \phi_2(t,m),\alpha \ra$, where $\phi_i = b_i,\sigma_i,\tsg_i$ resp., for $i=0,1,2$, are functions defined on $[0,T]\times \Ptwo(\R^n)$ with $\phi_1,\phi_2$ bounded and $\phi_0$ satisfies
\begin{equation*}
	| \phi_0(t,m) | \leq K\l(1+\l( \int_{\R^n} |y|^2 dm(y) \r)^{\frac{1}{2}}\r).
\end{equation*}

\end{assumpC}
\begin{assumpC}\label{lip_x}  $\d{x}f$, $\d{\alpha}f$, $\d{x}g$ exist and are $K$-Lipschitz continuous in $(x,\alpha)$ uniformly in $(t,m)$. 
\end{assumpC}
\begin{assumpC}\label{growth} $f,g$ satisfy a quadratic growth condition in $m$ and $\d{x}f,\d{\alpha}f,\d{x}g$ satisfy a linear growth condition in $(x,\alpha,m)$. That is, for any $t \in [0,T], x \in \R^n, \alpha \in \R^k, m \in \Ptwo(\R^n)$, 
\begin{equation}\label{quadratic_growth}
	 \max\{|f(t,0,m,0)|, |g(0,m)|\}  \leq K \l(1 + \int_{\R^n} |y|^2 dm(y)\r),
\end{equation}
\begin{equation}
	\label{linear_growth}
	\begin{aligned}
		&\max\{|\d{\alpha}f(t,x,m,\alpha)|,	|\d{x}f(t,x,m,\alpha)|, |\d{x}g(x,m)|\}\\
		&\quad \leq K\l(1+|x|+|\alpha|+ \l(\int_{\R^n} |y|^2 dm(y)\r)^{\frac{1}{2}} \r).
	\end{aligned}	  
\end{equation}
		
%
\end{assumpC}

\begin{assumpC}\label{convex}\label{mfg1_last_ass} $g$ is convex in $x$ and $f$ is convex jointly in $(x,\alpha)$ with strict convexity in $\alpha$. That is, for any $x,x' \in \R^n, m \in \Ptwo(\R^n)$, 
	\begin{equation}\label{convexity}	
		\la \d{x}g(x,m)-\d{x}g(x',m) , x-x' \ra \geq 0
	\end{equation}
and there exists a constant $c_f > 0$ such that for any $t \in [0,T], x,x' \in \R^n, \alpha,\alpha' \in \R^k, m \in \Ptwo(\R^n)$,
	\begin{equation}
		\label{convexity_f}
		\begin{aligned}
			f(t,x',m,\alpha') &\geq f(t,x,m,\alpha) + \la \d{x}f(t,x,m,\alpha),x'-x \ra\\
			&\quad + \la \d{\alpha}f(t,x,m,\alpha) , \alpha'-\alpha \ra + c_f|\alpha'-\alpha|^2.
		\end{aligned}
	\end{equation}
\end{assumpC}

The Lipschitz and linear growth conditions \ref{lip_x}, \ref{growth} are standard assumptions to ensure the existence of a strong solution.
 The linear-convex assumptions \ref{linear}, \ref{convex} are essential to our setup in various ways. First, they ensure that the Hamiltonian is strictly convex, so that there is a unique minimizer in a feedback form. In addition, they satisfy sufficient conditions for the SMP so that solving an optimal control problem can be translated to solving the corresponding FBSDE. See section 6.4.2 in \cite{pham2009} for instance. Lastly, they give a monotone property for the FBSDE corresponding to an individual player control problem \rref{ind_formulation} so that it is uniquely solvable. See  \cite{peng1999,hu1990} for well-posedness result of FBSDEs related to convex control problems.

The second set of assumptions are conditions on the $m$-argument in the cost functions. These assumptions are essential in showing the wellposed-ness of MFG with common noise.

\begin{assumpC}\label{no_m}\label{mfg2_first_ass}The functions $b,\sigma,\tsg$ are independent of m. 
\end{assumpC}

\begin{assumpC}(Lipschitz in $m$)\label{lip_m} $\d{x}g,\d{x}f$ is Lipschitz continuous in $m$ uniformly in $(t,x)$, i.e. there exists a constant $K$ such that
\begin{equation}
\begin{aligned}
	 &|\d{x}g(x,m)-\d{x}g(x,m')|  \leq K \W_2(m,m') \\
	 &|\d{x}f(t,x,m,\alpha)-\d{x}f(t,x,m',\alpha)|  \leq K \W_2(m,m')\\
\end{aligned}
\end{equation}
for all $t \in [0,T], x \in \R^n,\alpha \in \R^k, m,m' \in \Ptwo(\R^n)$, where $\W_2(m,m')$ is the second order Wasserstein metric defined by (\ref{wass}). 

\end{assumpC}

\begin{assumpC}(Separable in $\alpha,m$)\label{sep} $f$ is of the form
	\begin{equation}\label{sep_f}
		f(t,x,m,\alpha) = f^0(t,x,\alpha)+f^1(t,x,m)
	\end{equation}
	where $f^0$ is assumed to be convex in $(x,\alpha)$ strictly in $\alpha$, $f^1$ is assumed to be convex in $x$.
\end{assumpC}

\begin{assumpC}(Weak monotonicity)\label{mon}\label{mfg2_last_ass}
 For all $t \in [0,T]$, $m,m' \in \Ptwo(\R^n)$ and $\gamma \in \P_2(\R^{2n})$ with marginals $m,m'$ respectively,
\begin{equation}
\begin{aligned}
		 &\int_{\R^2} \l[ \la \d{x}g(x,m)-\d{x}g(y,m') ,  x-y \ra \r] \gamma(dx,dy) \geq 0  \\
		& \int_{\R^2} \l[ \la \d{x}f(t,x,m,\alpha)-\d{x}f(t,y,m',\alpha) ,  x-y \ra \r] \gamma(dx,dy) \geq 0 
\end{aligned}
\end{equation}
	Equivalently, for any $x \in \R^n, \xi,\xi'  \in \Ltwo(\b{\Omega},\b{\F},\b{\mb{P}};\R^n)$ where $(\b{\Omega},\b{\F},\b{\mb{P}})$ is an arbitrary probability space,
\begin{equation}\label{monotone}\begin{aligned}
& \b{\mb{E}}[  \la \d{x}g(\xi,\mb{P}_\xi) - \d{x}g(\xi',\mb{P}_{\xi'} ),  \xi-\xi'  \ra ] \geq 0  \\
&		\b{\mb{E}}[  \la \d{x}f(t,\xi,\mb{P}_\xi,\alpha) - \d{x}f(t,\xi',\mb{P}_{\xi'},\alpha) , \xi-\xi' \ra] \geq 0 
\end{aligned}
\end{equation}
	
\end{assumpC}



Assumption \ref{mfg1_first_ass}-\ref{mfg1_last_ass} are similar to those used in \cite{carmona2013probabilistic} to apply the SMP to MFG without common noise. To establish existence result, in addition to \ref{mfg1_first_ass}-\ref{mfg1_last_ass}, they assume \ref{lip_m} and what they refer to as a \textit{weak mean reverting} assumption. The latter states that there exists a constant $C > 0$ such that for all $t \in [0,T], x \in \R^n$
\begin{equation}\label{weakMR}
\begin{aligned}
&\la x ,\d{x}f(t,0,\delta_x,0) \ra  \geq -C(1+|x|) \\
 &\la x , \d{x}g(0,\delta_x) \ra  \geq -C(1+|x|)
 \end{aligned}
 \end{equation} 
 where $\delta_x$ denotes the Dirac measure at $x$. By plugging in deterministic $\xi = x,\xi' = 0$ in \rref{monotone}, we can see that the weak monotonicity assumption \ref{mon} is a stronger version of \rref{weakMR}. The weak monotone condition was first introduced in \cite{ahuja2014} for the terminal cost to obtain wellposed-ness result for MFG with common noise under linear state process and quadratic running cost. Our result here extends it to cover a more general running cost function. 

Note that the separability condition \ref{sep} is not necessary for existence of a solution of MFG without common noise, but is only needed for the uniqueness result. See Proposition 3.7 and 3.8 in \cite{carmona2013probabilistic} for instance. In our case, we rely on the monotone property of the mean-field FBSDE and this condition is necessary to obtain this property.

%
%

For the uniqueness result, the main assumptions in the literature \cite{carmona2013probabilistic,cardaliaguet2010,gomes2013survey} are the separability in the control and mean-field term (assumption \ref{sep}) and the Lasry and Lions' monotonicity property which states that
\begin{equation*}
	\int (h(x,m_1)-h(x,m_2))d(m_1-m_2)(x) \geq 0
\end{equation*}
for any $m_1,m_2 \in \Ptwo(\R^n)$. This condition can be expressed in terms of random variables as follows; For any $\xi,\xi' \in \Ltwo(\h{\Omega},\h{\F},\h{\mb{P}};\R^n)$ where $(\h{\Omega},\h{\F},\h{\mb{P}})$ is an arbitrary probability space. 
\begin{equation}\label{monLion}
 	 \h{\mb{E}}\l[ h(\xi',\mb{P}_{\xi'}) +h(\xi,\mb{P}_\xi) - h(\xi,\mb{P}_{\xi'}) - h(\xi',\mb{P}_\xi) \r] \geq 0
\end{equation}
Our weak monotonicity assumption \ref{mon} is, as the name suggests, a weaker version of \rref{monLion} when the cost functions are convex. See Lemma 4.2 in \cite{ahuja2014}. The converse of the proposition above does not hold as seen from the examples below (when $n=1$). 

%
%

\be\label{e1}
 \begin{gathered}
 	f(t,x,m,\alpha) = A\alpha^2 + B \l( x - \int z dm(z) \r)^2, \quad g(x,m) = C\l( x - \int z dm(z) \r)^2, 
\end{gathered}
\ee
or
\begin{gather*}
 	f(t,x,m,\alpha) =A\alpha^2 + B \int (x-z)^2 dm(z), \quad g(x,m) = C\int (x-z)^2 dm(z),
\end{gather*}
 where $A,B,C > 0$. As a result, we have given a more general uniqueness for MFG without common noise. These cost functions occur frequently in applications (see \cite{carmona2013mean,gueant2011} for instance). A similar example of cost functions satisfying our assumptions includes the general linear-quadratic mean-field games (LQMFG) discussed in \cite{bensoussan2014linear} where $f,g$ take the form
 \begin{equation}
\begin{aligned}
& f(t,x,m,\alpha) = \frac{1}{2}\l( qx^2 + \alpha^2 + \bq(x-s\bm)^2 \r) \\
& g(x,m) = \frac{1}{2} \l( q_T x^2 + (x-s_T \bm)^2\bq_T \r)
\end{aligned}
 \end{equation}
 where  $\bm = \int_\R zdm(z) $ and $q,\bq,s,q_T,\bq_T,s_T$ are constant satisfying $ q+\bq-\bq s \geq 0, q_T+\bq_T-\bq_Ts_T \geq 0 $.

\subsection{Existence and uniqueness}\label{subsec_mfg_wellposed} We begin by discussing the SMP for MFG with common noise. Given a stochastic flow of probability measure $m=(m_t)_{0 \leq t \leq T} \in \CP{n}$, we define the \textit{generalized Hamiltonian}
\begin{equation}
	\label{Hamiltonian_s}
	\begin{aligned}
		H(t,a,x,y,z,\tz,m) &\triangleq \la b(t,x,m,a ),y \ra + \la \sigma(t,x,m,a),z \ra\\
		&\quad + \la \tsg(t,x,m,a),\tz\ra + f(t,x,m,a).
	\end{aligned}
\end{equation}
Under assumption \ref{mfg1_first_ass}-\ref{mfg1_last_ass}, the generalized Hamiltonian is strictly convex in the control argument and has a unique minimizer
\begin{equation*}
	\ba: [0,T]  \times \R^{n}  \times \R^n \times \R^{n }  \times \R^{n } \times \Ptwo(\R^n) \to \R^k
\end{equation*}
We then define the \textit{Hamiltonian}
\begin{align*}
	\bH(t,x,y,z,\tz,m) &= \min_{a \in \R^k} H(t,a,x,y,z,\tz,m)\\
	&= H(t,\ba(t,x,y,z,\tz,m),x,y,z,\tz,m)
\end{align*}
and define $(\bb,\bsg,\btsg)(t,x,y,z,\tz,m)$ similarly. It is easy to check that
\begin{equation*}
	\d{x}H(t,\ba(t,x,y,z,\tz,m),x,y,z,\tz,m) = \d{x}\bH(t,x,y,z,\tz,m)
\end{equation*}
Next, consider the system of forward backward stochastic differential equation (FBSDE)
\begin{equation}
	\label{fbsde_ind}
	\begin{aligned}
		dX_{t} &= \bb(t,X_t,Y_t,Z_t,\tZ_t,m_t) dt + \bsg(t,X_t,Y_t,Z_t,\tZ_t,m_t) dW_{t}\\
		&\quad + \btsg(t,X_t,Y_t,Z_t,\tZ_t,m_t) \tdW_{t} \\
		dY_{t}  &=  -\d{x}\bH(t,X_t,Y_t,Z_t,\tZ_t,m_t)dt+Z_{t} dW_{t} + \tZ_{t}\tdW_{t} \\
		X_{0} &= \xi_{0}, \quad Y_{T} = \d{x}g(X_{T},m_T) 
	\end{aligned}
\end{equation}
We now state the SMP for an individual control problem given 
\begin{equation*}
	m \in \CP{n}
\end{equation*}
in term of FBSDE \rref{fbsde_ind}. 

\begin{thm}\label{smp_individual}
Assume that \ref{mfg1_first_ass}-\ref{mfg1_last_ass} holds, let
\begin{equation*}
	m=(m_t)_{0 \leq t \leq T} \in \CP{n},
\end{equation*}
then the individual control problem given $m$ has an optimal 
control 
\begin{equation*}
	\ha_t \in \Htwo{k}
\end{equation*}
if and only if FBSDE \rref{fbsde_ind} has an adapted solution $(X_t,Y_t,Z_t,\tZ_t)_{0 \leq t \leq T}$ satisfying
\begin{equation*}
	\mb{E}\l[ \sup_{0 \leq t \leq T} [ |X_t|^2 + |Y_t|^2 ] + \int_0^T [|Z_t|^2 + |\tZ_t|^2] dt \r] < \infty .
\end{equation*}
In that case, the optimal control is given by 
\begin{equation*}
	\ha_t = \ba(t,X_t,Y_t,Z_t,\tZ_t,m_t), \quad\forall t \in [0,T]
\end{equation*}
\end{thm}

\begin{proof} Given $m=(m_t)_{0 \leq t \leq T} \in \CP{n}$, then an individual control problem given $m$ is simply a classical control problem with random coefficients and cost functions. The result then follows from the SMP for linear-covex control with random coefficients (see Theorem 3.2 in \cite{karatzas1995}).
\end{proof}


The definition of a MFG solution states that given the stochastic flow of probability measure $m^\alpha \in \CP{n}$ corresponding to a control $\alpha \in \Htwo{k}$, the optimal control of an individual control problem given $m^\alpha$ is again $\alpha$. This definition is equivalent to the following \textit{consistency} condition
\begin{equation*}
	m^\alpha_t = \mb{P}_{X^\alpha_t | \tF_t}
\end{equation*}
Plugging this to \rref{fbsde_ind}, we have the SMP for \cMFG.

\begin{thm}[SMP for $\cMFG$]\label{smp_mfg}
Assume that \ref{mfg1_first_ass}-\ref{mfg1_last_ass} holds, then $\ha \in \Htwo{k}$ is a solution to MFG  if and only if the FBSDE
\begin{equation}
	\label{fbsde_mfg}
	\begin{aligned}
		dX_{t} &= \bb(t,X_t,Y_{t},Z_t,\tZ_t,\mb{P}_{X_t|\tF_t}) dt + \bsg(t,X_t,Y_{t},Z_t,\tZ_t,\mb{P}_{X_t|\tF_t}) dW_{t}\\
		&+ \btsg(t,X_t,Y_{t},Z_t,\tZ_t,\mb{P}_{X_t|\tF_t}) \tdW_{t} \\
		dY_{t} &=  -\d{x}\bH(t,X_t,Y_{t},Z_t,\tZ_t,\mb{P}_{X_t|\tF_t})dt+Z_{t} dW_{t} + \tZ_{t}\tdW_{t} \\
		X_{0} &= \xi_{0}, \quad Y_{T} = \d{x}g(X_{T},\mb{P}_{X_T|\tF_T})
	\end{aligned}
\end{equation}
has an adapted solution $(X_t,Y_t,Z_t,\tZ_t)_{0 \leq t \leq T}$ satisfying
\begin{equation*}
	\mb{E}\l[ \sup_{0 \leq t \leq T} [ |X_t|^2 + |Y_t|^2 ] + \int_0^T [|Z_t|^2 + |\tZ_t|^2] dt \r] < \infty .
\end{equation*}
In that case, a MFG solution is given by
\begin{equation*}
	\ha_t =  \ba(t,X_t,Y_t, Z_t,\tZ_t,\mb{P}_{X_t|\tF_t}), \qquad \forall t \in [0,T]
\end{equation*}
\end{thm}


Equation \rref{fbsde_mfg} was first introduced in \cite{carmona2013probabilistic} from the $\oMFG$ problem in which case the conditional law $\mb{P}_{X_t|\tF_t}$ is simply the law $\mb{P}_{X_t}$. In \cite{carmona2013probabilistic}, Carmona and Delarue, by using Schauder fixed point theorem, show that the mean-field FBSDE corresponding to a $\oMFG$ is solvable under assumptions similar to \ref{mfg1_first_ass}-\ref{mfg1_last_ass}, \ref{lip_m}, plus what they call a weak mean reverting assumptions (see \rref{weakMR}). However, the same proof cannot be extended to the case of common noise since we can no longer find an invariant compact subset. This is due to the fact that, in the case of common noise, we are dealing with a much larger space of \textit{stochastic} flow of probability measure instead of a \textit{deterministic} one. 

Since then, several work has been done that deal with the common noise models \cite{carmona2014commonnoise,ahuja2014,lacker2014translation}. In 
\cite{carmona2014commonnoise}, Carmona et al. considered the notion of \textit{weak} solution and, by finite-dimensional approximation of the common noise, proved 
its existence under a rather general set of assumptions. In \cite{lacker2014translation}, Lacker and Webster gives existence result under a class of 
\textit{translation invariant} MFG models. In \cite{ahuja2014}, Ahuja introduces a weak monotone assumption and prove well-posedness result for $\cMFG$ using the 
Banach fixed point theorem over small time interval and extend the result to arbitrary time duration. Our work here essentially gives an extension of \cite{ahuja2014} to a more general system by viewing it as part of a general class of 
monotone functional FBSDE.

We now discuss existence and uniqueness of solutions to \rref{fbsde_mfg} and thereby gives a well-posedness result of $\cMFG$. Using Theorem \ref{smp_mfg}, these results are mostly an application of the results from section \ref{sec_mkvfbsde}.

\begin{thm}\label{wellposed_mkfbsde} Assume \ref{mfg1_first_ass}-\ref{mfg2_last_ass} hold, then there exists a unique solution $(X_t,Y_t,Z_t,\tZ_t)_{s\leq t \leq T}$ to FBSDE \rref{fbsde_mfg} satisfying
\begin{equation}
\begin{aligned}
&\mb{E}\l[ \sup_{s \leq t \leq T} [ |X|_t^2 + |Y|_t^2 ] + \int_s^T [|Z|_t^2 + |\tZ|_t^2] dt \r] < \infty \\
\end{aligned}
\end{equation}
\end{thm}
 
\begin{proof} We need to verify that under \ref{mfg1_first_ass}-\ref{mfg2_last_ass}, the corresponding functions $(\bb,\bsg, \btsg, -\d{x}\bH, \d{x}g)$ of FBSDE \rref{fbsde_mfg} satisfies \ref{measurable_mkv}-\ref{lip_mon_mkv}. The result then follows from Theorem \ref{wellposed_mckean}.

First, using \ref{linear}, \ref{no_m}, \ref{sep} and optimal condition for $\ba$, it follows that
\begin{equation}\label{opt}
\begin{aligned}
0 &= \d{\alpha}H(t,\ba(t,x,y,z,\tz,m),x,y,z,\tz,m)  \\
& = \d{\alpha}b(t,x,m,\ba(t,x,y,z,\tz,m))^T y +  \d{\alpha}\sigma(t,x,m,\ba(t,x,y,z,\tz,m))^T z\\
&\quad + \d{\alpha}\tsg(t,x,m,\ba(t,x,y,z,\tz,m))^T \tz + \d{\alpha}f(t,x,m,\ba(t,x,y,z,\tz,m)) \\
& =  b_2(t)^T y + \sigma_2(t)^T z  + \tsg_2(t)^T \tz + \d{\alpha}f^0(t,x,\ba(t,x,y,z,\tz,m))
\end{aligned}
\end{equation}
This implies that $\ba$ is independent of $m$. From now, we write $\ba =\ba(t,x,y,z,\tz)$. We also have
\begin{equation}
	\label{opt2}
	\begin{aligned}
		\d{x}\bH(t,x,y,z,\tz,m) &= \d{x}H(t,\ba(t,x,y,z,\tz),x,y,z,\tz,m)\\
		&= b_1(t)^T y  +  \sigma_1(t)^T z  + \tsg_1(t)^T\tz  +  \d{x}f^1(t,x,m)\\
		&\qquad + \d{x}f^0(t,x,\ba(t,x,y,z,\tz))
	\end{aligned}
\end{equation}
Furthermore, by using strict convexity assumption \ref{convex},  we have
\be
\begin{aligned}
	f^0(t,x',\alpha') &\geq f^0(t,x,\alpha)+\la \d{x}f^0(t,x,\alpha), x'-x \ra\\
	&\quad +\la \d{\alpha}f^0(t,x,\alpha), \alpha'-\alpha \ra + c_f |  \alpha'-\alpha|^2\\
	f^0(t,x,\alpha) &\geq f^0(t,x',\alpha')+\la \d{x}f^0(t,x',\alpha'), x-x' \ra\\
	&\quad +\la \d{\alpha}f^0(t,x',\alpha'), \alpha-\alpha' \ra + c_f |  \alpha'-\alpha|^2.
\end{aligned}
\ee
Summing both equations yields
\begin{equation}
	\label{bound_alpha}
	\begin{aligned}
		 2c_f |  \alpha'-\alpha|^2 &\leq \la \d{x}f^0(t,x',\alpha')-\d{x}f^0(t,x,\alpha) , x'-x \ra  \\
		 &\quad+\la \d{\alpha}f^0(t,x',\alpha')-\d{\alpha}f^0(t,x,\alpha), \alpha'-\alpha \ra
	\end{aligned}
\end{equation}

Now we verify \ref{measurable_mkv}. From \rref{opt}, we have
\begin{equation*}
	\d{\alpha}f^0(t,0,\ba(t,0,0,0,0)) = 0
\end{equation*}
Combining with \rref{bound_alpha} using $x=x'=\alpha'=0$, it follows that
\be\label{bound_alpha0}
 \int_0^T |\ba(t,0,0,0,0)|^2 dt  \leq \frac{1}{c_f} \int_0^T |\d{\alpha}f^0(t,0,0)|^2 dt < \infty 
\ee
By assumption \ref{linear}, we then have
\begin{align*}
	\int_0^T |\bb(t, 0, 0, 0, 0, \delta_0)|^{2} dt &= \int_0^T |b(t,0,\delta_0,\ba(t,0,0,0,0))|^2 dt\\
	&= \int_0^T |b_0(t)+b_2(t)\ba(t,0,0,0,0)|^2 dt < \infty
\end{align*}
and similarly for $\bsg,\btsg$. The same bound holds for $\d{x}\bH,\d{x}g$ by \rref{opt2}, \rref{bound_alpha0}, and the linear growth assumption \ref{growth}. Thus,  \ref{measurable_mkv} holds as desired.

Next, by using \rref{opt} with  $(x,y,z,\tz),(x',y,z,\tz) \in \R^p$, taking the difference, and using \rref{bound_alpha}, it follows that $\ba$ is Lipschitz in $x$. Furthermore, by using \rref{opt} again with $(x,y,z,\tz),(x,y',z',\tz') \in \R^p$ and taking the difference, we get
\begin{align*}
	0 &= b_2(t)^T \D y +  \sigma_2(t)^T \D z  +  \tsg_2(t)^T\D \tz  +\d{\alpha}f^0(t,x,\ba(t,x,y,z,\tz))\\
	&\quad - \d{\alpha}f^0(t,x,\ba(t,x,y',z',\tz' ))
\end{align*}
Using \rref{bound_alpha} with $x'=x$ and Lipschitz assumption on $\d{\alpha}f^0$, we have
\begin{equation}
	\label{bound_diff_alpha}
	\begin{aligned}
		\frac{1}{K} |  b_2(t)^T \D y +  \sigma_2(t)^T \D z  +  \tsg_2(t)^T\D \tz  | 
		&\leq | \ba(t,x,y,z,\tz)) -\ba(t,x,y',z',\tz' )) |\\
		&\leq K |  b_2(t)^T \D y +  \sigma_2(t)^T \D z  +  \tsg_2(t)^T\D \tz  |
	\end{aligned}
\end{equation}
That is,
\begin{equation}
	\label{lip_ba}
	\begin{aligned}
		&| \ba(t,x,y,z,\tz)) -\ba(t,x',y',z',\tz' )) |\\
		&\quad \leq K \l(  |\Delta x | + | b_2(t)^T \D y +  \sigma_2(t)^T \D z  +  \tsg_2(t)^T\D \tz | \r).
	\end{aligned}
\end{equation}
Combining with \ref{linear},\ref{no_m}, and Lipschitz in $(x,m)$ of $f^1$ (\ref{lip_x},\ref{lip_m}), \ref{lip_mon_mkv}\rref{lip_mkv_a} then follows.

Lastly, we check the monotonicity condition \ref{lip_mon_mkv}\rref{mon_mkv_b}. For $(X,Y,Z,\tZ)$, $(X',Y',Z',\tZ') \in \Ltwo(\h{\Omega},\h{\F},\h{\mb{P}};\R^p)$, 
we have
\be\label{all_eq}
\begin{aligned}
 -\la \D \d{x}\bH_t, \D X \ra  &=- \la b_1(t)^T \D Y, \D X \ra -  \la \sigma_1(t)^T\D Z, \D X \ra - \langle \tsg_1(t)^T\D \tZ, \D X \rangle\\
 &\quad - \la \D\d{x}f^0_t, \D X\ra -\la \D\d{x}f^1_t, \D X \ra  \\
  \la \D \bb_t, \D Y \ra &= \la b_1(t)\D X, \D Y \ra + \la b_2(t)\D \ba , \D Y \ra  \\
 \tr \la \D \bsg_t, \D Z \ra &= \la \sg_1(t)\D X, \D Z \ra + \la \sg_2(t)\D \ba , \D Z \ra \\
   \tr \langle \D \btsg_t, \D \tZ \rangle &= \langle \tsg_1(t)\D X, \D \tZ \rangle + \langle \tsg_2(t)\D \ba , \D \tZ \rangle 
\end{aligned}
\ee
where $\D \d{x}\bH_t = \d{x}\bH(t,X',Y',Z',\tZ')-\d{x}\bH(t,X,Y,Z,\tZ), \D X = X'-X$ and similarly for other terms. From \rref{opt} and \rref{bound_alpha}, we have
\be\label{bound_f0}
  -\la \D \d{x}f^0_t , \Delta X \ra + \la b_2(t)^T \D Y +  \sigma_2(t)^T \D Z  +  \tsg_2(t)^T\D \tZ, \Delta \ba \ra + 2 c_f| \D \ba | ^2 \leq 0
 \ee
Moreover, by the weak monotonicity assumption \ref{mon}, we have
\be\label{mon_fg_2}
\begin{aligned}
&\h{\mb{E}}\l[ \la \D \d{x}f^1_t, \D X \ra  \r] \geq 0,\quad \h{\mb{E}}\l[ \la \D \d{x}g , \D X \ra \r]  \geq 0 
\end{aligned}
\ee
Combining \rref{bound_diff_alpha},\rref{all_eq},\rref{bound_f0}, and \rref{mon_fg_2} yields \ref{lip_mon_mkv}\rref{mon_mkv_b} as desired.

\end{proof}

From Theorem \ref{smp_mfg},\ref{wellposed_mkfbsde}, we have the wellposedness result for $\cMFG$ with common noise.

\begin{corollary}[Wellposedness of $\cMFG$] Under assumption \ref{mfg1_first_ass}-\ref{mfg2_last_ass}, there exists a unique $\cMFG$ solution for any initial $\xi_0 \in \Ltwo_{\F_0}$.\end{corollary}

\subsection{Markov property}


In the previous section, we seek an \textit{admissible} control or strategy in the space $\Htwo{k}$ which solves mean-field games with common noise \rref{mfg_formulation}. We show that a solution exists under linear-convex setting and weak monotone cost functions using the stochastic maximum principle. By using this approach, the control is given in an \textit{open-loop} form, that is, as a function of paths $(W_t,\tW_t)_{0 \leq t \leq T}$, which is often not desirable for practitioners as, in most cases, they are not easily observable compared to the state process $(X_t)_{0 \leq t \leq T}$. 

In a classical control problem, one can get the \textit{closed-loop} or \textit{feed-back} control, that is, as a function of state variables, by using the dynamic programming principle (DPP) approach instead. This method requires solving the Hamilton-Jacobi-Bellman (HJB) equation to obtain the value function and the corresponding optimal control as a function of time and state variables. We can obtain similar result for MFG in the absence of common noise. In that case, the flow of the controlled process under a MFG solution is \textit{deterministic}. As a result, the solution is simply an optimal control of a classical Markovian control problem and, thus, can be written in a feed-back form.

However, this property is not trivial in the case of common noise where the flow is now \textit{stochastic}. In this last section, we would like to show, as an application of the result from section \ref{subsec_decoupling}, that the control is indeed in \textit{closed-loop} or \textit{feed-back} form if we include the conditional law of the state variables. That is, it can be written as a \textit{deterministic} function of state variables and its conditional law thereby establishing the Markov property of MFG with common noise. Our main result is the following

\begin{thm} Assume that \ref{mfg1_first_ass}-\ref{mfg2_last_ass} holds and $\sigma_2(t)=\tsg_2(t)=0$, then the solution $(\ha_t)_{0 \leq t \leq T}$ to MFG with common noise \rref{mfg_formulation} is of the form
\be\label{ha_u}
\ha_t =  u(t,X_t, \mb{P}_{X_t | \tF_t})
\ee
where $u:[0,T] \times \R^n \times \Ptwo(\R^n)$ is a $K$-Lipschitz deterministic function.
\end{thm}


\begin{proof} 
From Theorem \ref{smp_mfg}, \ref{wellposed_mkfbsde}, we have shown that the solution to MFG with common noise \rref{mfg_formulation} is given by 
\be\label{opt_ba}
\ha_t = \ba(t,X_t,Y_t,Z_t,\tZ_t,\mb{P}_{X_t | \tF_t})
\ee
where $(X_t,Y_t,Z_t,\tZ_t)_{0 \leq t \leq T}$ is a solution to mean-field FBSDE \rref{fbsde_mfg} and $\ba$ is a deterministic function. Note that even though $\ba$ is deterministic, it does not imply the Markov property or feedback control form as the processes $(Y_t,Z_t,\tZ_t)_{0 \leq t \leq T}$ are not necessarily a \textit{deterministic} function of $X_t$. 

From the assumption $\sigma_2(t)=\tsg_2(t)=0$ and \rref{opt}, we have that $\ba$ is independent of $z,\tz$. We would like to apply Theorem \ref{ue_existence}, so we need to verify that assumption \ref{measurable_mkv}-\ref{lip_m_given} holds for $(\bb,\bsg, \btsg, -\d{x}\bH, \d{x}g)$. We have already shown that \ref{measurable_mkv}-\ref{lip_mon_mkv} holds in the proof of Theorem \ref{wellposed_mkfbsde}. 
\ref{onlyxy_mkv} immediately holds from \rref{fbsde_mfg}, and \ref{linearGrowthOnM_m_given} also holds directly from \ref{linear} and \ref{growth}. For \ref{lip_m_given},  the proof is nearly identical to that of \ref{lip_mon_mkv} in Theorem \ref{wellposed_mkfbsde}, but here we do not need the weak monotonicity condition for \rref{mon_fg_2} and use \ref{convex} instead.

Thus, by Theorem \ref{ue_existence}, there exists a deterministic $K$-Lipschitz function $U:[0,T] \times \R^n \times \Ptwo(\R^n)$ such that
\be\label{decouple_yt}
 Y_t = U(t,X_t,\mb{P}_{X_t | \tF_t})
 \ee
Let $u(t,x,m) = \ba(t,x,U(t,x,m),m)$, then the result follows from \rref{opt_ba}, \rref{decouple_yt}. That is,
\begin{equation*}
	\ha_t = \ba(t,X_t,Y_t,Z_t,\tZ_t,\mb{P}_{X_t | \tF_t}) = \ba(t,X_t,U(t,X_t,\mb{P}_{X_t | \tF_t}),\mb{P}_{X_t | \tF_t})
\end{equation*}
The Lipschitz property of $u$ follows from \rref{lip_ba} and Lemma \ref{estimate_ue}.
\end{proof}

For a classical stochastic control problem, the feed-back form optimal control is related to the gradient of the value function of the HJB equation. Similarly for the MFG, the function $U$ here is related to the gradient of the solution of the so-called \textit{master equation}, an infinite-dimensional second order PDE involving the space of probability measures. For interested readers, we refer to \cite{ahuja2015mean,carmona2014master, delarue2014classical,bensoussan2014master} for detail on the dynamic programming principle approach for MFG and discussions on the master equation.

\section{Acknowledgement}       

This research received no specific grant from any funding agency in the public, commercial, or not-for-profit sectors.

\appendix

\section{Proof of Theorem \ref{wellposed}}
\label{proof_wellposed}

Suppose $(\Delta X_t,\Delta Y_t, \Delta Z_t)_{0 \leq t \leq T}$ denote the difference of two solutions. By taking \ito\;lemma on $\la \Delta X_t, \Delta Y_t \ra $ and using \ref{lip_mon_bfg}\eqref{mon_bfg}, we have 
$$ \int_0^T \l| c^{(1)}_t(\Delta Y_t) + c^{(2)}_t(\Delta Z_t) \r|^2 dt \leq 0 $$
Then by \ref{lip_mon_bfg}\eqref{lip_bfg}, we have the uniqueness as desired. For existence, consider the FBSDE
 \begin{equation}
 \label{fbsde_alpha}
 \begin{aligned}
 	 dX_t &= \l[ \alpha B(t,X_t,Y_t,Z_t) + (1-\alpha) \l(-\bar{c}^{(1)}_t \l(c^{(1)}_t( Y_t) + c^{(2)}_t( Z_t) \r) \r) + \phi_t \r] dt  \\
	 &\quad  +\l[ \alpha \Sigma(t,X_t,Y_t,Z_t) + (1-\alpha) \l(-\bar{c}^{(2)}_t \l(c^{(1)}_t( Y_t) + c^{(2)}_t( Z_t) \r) \r) + \psi_t \r]  dW_t  \\
	 dY_t &=  \l[ \alpha F(t,X_t,Y_t,Z_t) + (1-\alpha)(-X_t) + \gamma_t \r] dt+Z_tdW_t \\
X_0 &= \xi, \quad Y_T = \alpha G(X_T) + \eta
\end{aligned}
\end{equation}
where $\bar{c}^{(1)}_t,\bar{c}^{(2)}_t$ are the adjoint operators of the bounded operator $c^{(1)}_t,c^{(2)}_t$, $(\phi_t,\psi_t,\gamma_t)_{0 \leq t \leq T}$ $\in \H_\mb{F}^2([0,T];\R^n)$, and $\eta \in \Ltwo_{\F_T}(\R^n)$. We will show that FBSDE \eqref{fbsde_alpha} with $\alpha=1$ has a unique solution for any $(\phi,\psi,\gamma,\eta)$ by showing that
\begin{enumerate}[(i)]
\item\label{first} FBSDE \eqref{fbsde_alpha} with $\alpha=0$ has a unique solution for any $(\phi,\psi,\gamma,\eta)$. 
\item\label{second}There exist $\delta_0 > 0$ such that for any $\alpha_0 \in [0,1)$, if FBSDE  \eqref{fbsde_alpha} with $\alpha = \alpha_0$ has a unique solution for any $(\phi,\psi,\gamma,\eta)$, then so does FBSDE  \eqref{fbsde_alpha} with $\alpha \in [\alpha_0,\alpha_0+\delta_0)$.
\end{enumerate}  

For \eqref{second}, we define a map $\Phi: \H^2([0,T];\R^n) \ni (x_t,y_t,z_t)_{0 \leq t \leq T}$ $\to$ $(X_t,Y_t,Z_t)_{0 \leq t \leq T} \in \H^2([0,T];\R^n)$ where $(X_t,Y_t,Z_t)_{0 \leq t \leq T}$ is a solution to 
 \begin{equation}
 \label{fbsde_alpha_0}
 \begin{aligned}
 	 dX_t &= \l[ \alpha_0 B(t,X_t,Y_t,Z_t) + (1-\alpha_0) \l(-\bar{c}^{(1)}_t \l(c^{(1)}_t( Y_t) + c^{(2)}_t( Z_t) \r) \r) + \phi_t \r] dt  \\
	 &\quad +  \delta \l[ B(t,x_t,y_t,z_t) +\bar{c}^{(1)}_t \l(c^{(1)}_t( y_t) + c^{(2)}_t( z_t) \r)  \r] dt \\
	 &\quad  +\l[ \alpha_0 \Sigma(t,X_t,Y_t,Z_t) + (1-\alpha_0) \l(-\bar{c}^{(2)}_t \l(c^{(1)}_t( Y_t) + c^{(2)}_t( Z_t) \r) \r) + \psi_t \r]  dW_t  \\
	 &\quad +  \delta \l[ \Sigma(t,x_t,y_t,z_t) +\bar{c}^{(2)}_t \l(c^{(1)}_t( y_t) + c^{(2)}_t( z_t) \r)  \r] dW_t \\
	 dY_t &=  \l[ \alpha_0 F(t,X_t,Y_t,Z_t) + (1-\alpha_0)(-X_t) + \gamma_t + \delta \l(F(t,x_t,y_t,z_t) + x_t \r) \r] dt+Z_tdW_t \\
X_0 &= \xi, \quad Y_T = \alpha_0 G(X_T) +\delta G(x_T)+ \eta
\end{aligned}
\end{equation}
The map is well-defined by assumption in \eqref{second} for $\alpha=\alpha_0$. Then it can be shown that for sufficiently small $\delta_0 > 0$ depending only on the Lipschitz constant $K$ and time duration $T$, $\Phi$ is a contraction for all $\delta \leq \delta_0$; the proof is identical to that of Theorem 3.1 in \cite{peng1999}, so we omit it here.

For \eqref{first}, we need to apply method of continuation again by considering the FBSDE 
 \begin{equation}
 \label{fbsde_alpha_1}
 \begin{aligned}
 	 dX_t &= \l[ \alpha \l(-\bar{c}^{(1)}_t \l(c^{(1)}_t( Y_t) + c^{(2)}_t( Z_t) \r) \r)  + \phi_t \r] dt  \\
	 &\quad  +\l[\alpha \l(-\bar{c}^{(2)}_t \l(c^{(1)}_t( Y_t) + c^{(2)}_t( Z_t) \r) \r) + \psi_t \r]  dW_t  \\
	 dY_t &=  \l[ -X_t + \gamma_t \r] dt+Z_tdW_t \\
X_0 &= \xi, \quad Y_T = \eta.
\end{aligned}
\end{equation}
We aim to show that
\begin{enumerate}[(i)]
\setcounter{enumi}{2}
\item\label{third} FBSDE \eqref{fbsde_alpha_1} with $\alpha=0$ has a unique solution for any $(\phi,\psi,\gamma,\eta)$. 
\item\label{fourth}There exist an $\delta_1 > 0$ such that for any $\alpha_1 \in [0,1)$, if FBSDE  \eqref{fbsde_alpha_1} with $\alpha = \alpha_1$ has a unique solution for any $(\phi,\psi,\gamma,\eta)$, then so does FBSDE  \eqref{fbsde_alpha} with $\alpha \in [\alpha_1,\alpha_1+\delta_1)$.
\end{enumerate}  

\eqref{third} follows from Lemma 2.5 in \cite{peng1999} (with $G=I,\beta_1=1,\beta_2=0$). For \eqref{fourth}, we proceed similarly by defining a map $\Phi: \H^2([0,T];\R^n) \ni (x_t,y_t,z_t)_{0 \leq t \leq T} \to (X_t,Y_t,Z_t)_{0 \leq t \leq T} \in \H^2([0,T];\R^n)$ where $(X_t,Y_t,Z_t)_{0 \leq t \leq T}$ is a solution to
 \begin{equation}
 \label{fbsde_alpha_2}
 \begin{aligned}
 	 dX_t &= \l[ \alpha_1 \l(-\bar{c}^{(1)}_t \l(c^{(1)}_t( Y_t) + c^{(2)}_t( Z_t) \r) \r)  + \phi_t + \delta\l(-\bar{c}^{(1)}_t \l(c^{(1)}_t( y_t) + c^{(2)}_t( z_t) \r) \r)  \r] dt  \\
	 & +\l[\alpha_1 \l(-\bar{c}^{(2)}_t \l(c^{(1)}_t( Y_t) + c^{(2)}_t( Z_t) \r) \r) + \psi_t +\delta \l(-\bar{c}^{(2)}_t \l(c^{(1)}_t( y_t) + c^{(2)}_t( z_t) \r) \r) \r]  dW_t  \\
	 dY_t &=  \l[ -X_t + \gamma_t \r] dt+Z_tdW_t \\
X_0 &= \xi, \quad Y_T = \eta
\end{aligned}
\end{equation}
The map is well-defined by assumption in \eqref{fourth} for $\alpha=\alpha_1$. Similarly, it can be shown that for sufficiently small $\delta_1 > 0$ depending only on the Lipschitz constant $K$ and time duration $T$, $\Phi$ is a contraction for all $\delta \leq \delta_1$; the proof is identical to that of Lemma 2.4 in \cite{peng1999}, so we omit it here.

\section{Proof of Theorem \ref{estimate_fbsde}}
\label{proof_estimate_diff}

We will use the following notations in this proof; for $\Phi = B, \Sigma, F$
\begin{equation*}
\begin{aligned}
&\D \Phi_t = \Phi(t,\theta)-\Phi(t,\theta'), &\b{\Phi} = (\Phi-\Phi')(t,\theta'_t)\\
&\D G = G(X_T)-G(X_T'), &\b{G} = (G-G')(X_T')
\end{aligned}
\end{equation*}
By Ito's lemma on $\la\Delta X_{t},\Delta Y_{t} \ra$,
\begin{equation*}
\begin{split}
&\mb{E}[\Ind_{A} \la \Delta X_{s},\Delta Y_{s} \ra ]  \\
&= \mb{E}[\Ind_{A}\la \Delta X_{T}, \Delta G + \bar{G} \ra]  \\
&\quad- \mb{E}[\Ind_{A} \int^{T}_{s}\left( \la \Delta F_{t} + \bar{F}_{t},\Delta X_{t} \ra + \la \Delta B_{t} +\bar{B}_{t} , \Delta Y_{t} \ra +   \la \Delta \Sigma_{t} + \bar{\Sigma}_{t}, \Delta Z_{t} \ra \right)dt]\\
& \geq \mb{E}\Ind_{A}  \int^{T}_{s}\beta|c^{(1)}_{t}(\D Y_t) + c^{(2)}_{t}(\D Z_t )|^{2}dt + \mb{E}\Ind_{A}\la \Delta X_{T},\bar{G} \ra  \\
& \quad - \mb{E}[\Ind_{A} \int^{T}_{s}( \la \bar{F}_{t},\Delta X_{t} \ra +\la  \Delta Y_{t} ,\bar{B}_{t} \ra + \la \Delta Z_{t}, \bar{\Sigma}_{t} \ra )dt] 
\end{split}
\end{equation*}
Here we used assumption \ref{lip_mon_bfg}\rref{mon_bfg} and the fact that  A $\in$ ${\mathcal{G}}_{s}$ which helps eliminate the stochastic integral after taking the expectation. Thus, we have 
\begin{equation}\label{bdY_cdZ_estimate}
\begin{aligned}
&\mb{E}\Ind_{A}  \int^{T}_{s}|c^{(1)}_{t}(\D Y_t) + c^{(2)}_{t}(\D Z_t )|^{2}dt \\
&\leq \mb{E}[\Ind_{A} \la \Delta \xi, \Delta Y_{s} \ra ] +\mb{E}\Ind_{A}\la \Delta X_{T},\bar{G} \ra     \\
&\quad + \mb{E}[\Ind_{A} \int^{T}_{s}( \la \bar{F}_{t},\Delta X_{t} \ra +\la  \Delta Y_{t} ,\bar{B}_{t} \ra + \la \Delta Z_{t}, \bar{\Sigma}_{t} \ra )dt] 
\end{aligned}
\end{equation}
Next, by applying Ito lemma on $|\Delta Y_{t}|^{2}$ and using standard argument involving $K$-Lipschitz property of $F,G$, Young's inequality, Burkholder-Davis-Gundy (BDG) inequality (see Theorem 3.28 in \cite{karatzas2012brownian}), and Gronwall inequality, we have 

\begin{equation}\label{Y Z estimate 1}
\begin{split}
&\mb{E}\Ind_{A}(\displaystyle\sup_{t\leq u\leq T}|\Delta Y_{u}|^{2} + \int^{T}_{t} |\Delta Z_{u}|^{2}du)\\
&\leq C_{K,T}\mb{E}\Ind_{A}\l(|\Delta X_{T}|^{2} + |\bar{G}|^{2} + \int^{T}_{t}(|\bar{F}_{u}|^{2} +  |\Delta X_{u}|^{2})du  \r)
\end{split}
\end{equation}
for some constant $C_{K,T} > 0$ depending only on $K,T$. Thus, we also have

\begin{equation}\label{Y first estimate}
\begin{split}
&\mb{E}\Ind_{A}(|\Delta Y_{t}|^{2} + \int^{T}_{t} |\Delta Z_{u}|^{2}du)\\
&\leq C_{K, T}\mb{E}\Ind_{A}\l( |\Delta X_{T}|^{2} + |\bar{G}|^{2} + \int^{T}_{t}(|\bar{F}_{u}|^{2} +  |\Delta X_{u}|^{2})du \r)
\end{split}
\end{equation}

\noindent We now need an estimate on $|\Delta X_{t}|^{2}$. By Ito's lemma, Young's inequality, assumption \ref{lip_mon_bfg}\rref{lip_bfg}, and Gronwall inequality, we have
\begin{equation}\label{X first estimate}
\mb{E}\Ind_{A}|\Delta X_{t}|^{2} \leq C_{K, T}\mb{E}\Ind_{A}\{|\Delta \xi|^{2} + \int^{t}_{s}[|\bar{B}_{u}|^{2} + |\bar{\Sigma}_{u}|^{2} + |c^{(1)}_{u}(\Delta Y_{u})+c^{(2)}_{u}(\Delta Z_{u})|^{2} ]du \}
\end{equation}
Plugging this into \eqref{Y Z estimate 1}, we have

\begin{equation}\label{Y Z estimate 2}
\begin{split}
&\mb{E}\Ind_{A}(\displaystyle\sup_{s\leq t\leq T}|\Delta Y_{t}|^{2} + \int^{T}_{s} |\Delta Z_{t}|^{2}dt)\\
&\leq K\mb{E}\Ind_{A}(|\bar{G}|^{2}  + |\Delta \xi|^{2} ) \\
&\quad + K\mb{E}\Ind_{A}\int^{T}_{s}\left[|\bar{F}_{t}|^{2} + |\bar{B}_{t}|^{2} + |\bar{\Sigma}_{t}|^{2} + |c^{(1)}_{t}(\Delta Y_{t})+c^{(2)}_{t}(\Delta Z_{t})|^{2}\right]dt  
\end{split}
\end{equation}
By BDG inequality, it follows that 
\begin{equation}\label{X BDG}
\begin{split}
&\mb{E}\Ind_{A}\displaystyle\sup_{s \leq t \leq T}\int^{t}_{s} 2\la (\Delta \Sigma_{u} + \bar{\Sigma}_{u})^T \Delta X_{u} ,dW_{u} \ra \\
&\leq C \mb{E}\Ind_{A}[\displaystyle\sup_{s\leq t \leq T}\int^{t}_{s}|(\Delta \Sigma_{u} + \bar{\Sigma}_{u})^T\Delta X_{u} |^{2}du]\\
&\leq \eps \mb{E}\Ind_{A}\displaystyle\sup_{s\leq t\leq T}|\Delta X_{t}|^{2} + C \mb{E}\Ind_{A}(\int^{T}_{s} |\Delta \Sigma_{t} + \bar{\Sigma}_{t} |^{2} dt)\\
&\leq \eps \mb{E}\Ind_{A}\displaystyle\sup_{s\leq t\leq T}|\Delta X_{t}|^{2} + C \mb{E}\Ind_{A}(\int^{T}_{s} [|\bar{\Sigma}_{t} |^{2}  + |\Delta X_{t}|^{2} +|c^{(1)}_{t}(\Delta Y_{t})+c^{(2)}_{t}(\Delta Z_{t})|^{2} ] d t)
\end{split}
\end{equation}
Next, by applying Ito's lemma on $|\Delta X_{t}|^{2}$, using \eqref{X first estimate} and \eqref{X BDG}, we have
\begin{equation}\label{X estimate}
\begin{split}
&\mb{E}\Ind_{A}\displaystyle\sup_{s\leq t \leq T}|\D X_{t}|^{2}\\
&\leq C_{K, T}\mb{E}\Ind_{A}\left\{|\Delta \xi|^{2} + \int^{T}_{s}[|\bar{B}_{t}|^{2} + |\bar{\Sigma}_{t}|^{2} + |c^{(1)}_{t}(\Delta Y_{t})+c^{(2)}_{t}(\Delta Z_{t})|^{2} ]dt \right\}
\end{split}
\end{equation}
The result then follows by grouping estimates \eqref{Y Z estimate 2} and \eqref{X estimate} and using \rref{bdY_cdZ_estimate} and Young's inequality.

\bibliographystyle{elsarticle-num}
\bibliography{reference.bib}

\end{document}